\documentclass[final]{siamltex}
\usepackage{subfigure}

\usepackage[left=1in,top=1in,right=1in,bottom=1in,letterpaper]{geometry}
\usepackage{setspace,url}
\usepackage{verbatim}

\usepackage{multirow}
\usepackage{verbatim}
\usepackage{longtable}

\usepackage[leqno]{amsmath}
\usepackage{amsxtra, amsfonts,amscd, amssymb, graphicx,epsf}
\usepackage{algorithm,algorithmic}

\numberwithin{equation}{section}

\newcommand{\be}{\begin{equation}}
\newcommand{\ee}{\end{equation}}
\newcommand{\ba}{\begin{array}}
\newcommand{\ea}{\end{array}}
\newcommand{\bea}{\begin{eqnarray}}
\newcommand{\eea}{\end{eqnarray}}
\newcommand{\beaa}{\begin{eqnarray*}}
\newcommand{\eeaa}{\end{eqnarray*}}

\newcommand{\half}{\frac{1}{2}}
\newcommand{\br}{\mathbb{R}}

\newcommand{\ACal}{\mathcal{A}}

\newcommand{\Diag}{\mbox{Diag}}

\newcommand{\etal}{{ et al. }}

\newcommand{\st}{\mbox{ s.t.}}
\newcommand{\mtn}{{m\times n}}
\newcommand{\sgn}{\mbox{sgn}}

\newcommand{\shrink}{\mbox{shrink}}
\newcommand{\Shrink}{\mbox{Shrink}}
\newcommand{\PCal}{\mathcal{P}}

\newtheorem{remark}[theorem]{Remark}

\begin{document}
\title{Accelerated Linearized Bregman Method}
\author{Bo Huang\footnotemark[2] \and Shiqian Ma\footnotemark[1] \footnotemark[2] \and Donald Goldfarb\footnotemark[2]}
\renewcommand{\thefootnote}{\fnsymbol{footnote}}
\footnotetext[1]{Corresponding author.}
\footnotetext[2]{Department of Industrial Engineering and Operations Research,
Columbia University, New York, NY 10027. Email:
\{bh2359,sm2756,goldfarb\}@columbia.edu. Research supported in part by
NSF Grants DMS 06-06712 and DMS 10-16571, ONR Grants N00014-03-0514 and
N00014-08-1-1118, and DOE Grants DE-FG01-92ER-25126 and
DE-FG02-08ER-25856.}

\maketitle
% note: updated some typos after submission.
\begin{center} June 21, 2011 \end{center}

\begin{abstract}
In this paper, we propose and analyze an accelerated linearized Bregman (ALB) method for solving the basis pursuit and related sparse optimization problems. This accelerated algorithm is based on the fact that the linearized Bregman (LB) algorithm is equivalent to a gradient descent method applied to a certain dual formulation. We show that the LB method requires $O(1/\epsilon)$ iterations to obtain an $\epsilon$-optimal solution and the ALB algorithm reduces this  iteration complexity to $O(1/\sqrt{\epsilon})$ while requiring almost the same computational effort on each iteration. Numerical results on compressed sensing and matrix completion problems are presented that demonstrate that the ALB method can be significantly faster than the LB method.
\end{abstract}



\section{Introduction}
In this paper, we are interested in the following optimization problem
\bea\label{general-convex-min}
\min_{x\in\br^n} J(x) \quad \st \quad  Ax=b,
\eea
where $A\in\br^\mtn$, $b\in\br^m$ and $J(x)$ is a continuous convex function. An important instance of \eqref{general-convex-min} is the so-called basis pursuit problem when $J(x):=\|x\|_1=\sum_{j=1}^n|x_j|$:
\bea\label{basis-pursuit}
\min_{x\in\br^n} \|x\|_1 \quad \st \quad  Ax=b.
\eea
Since the development of the new paradigm of compressed sensing \cite{Candes-Romberg-Tao-2006,Donoho-06}, the basis pursuit problem \eqref{basis-pursuit} has become a topic of great interest.  In compressed sensing, $A$ is usually the product of a sensing matrix $\Phi$ and a transform basis matrix $\Psi$ and $b$ is a vector of the measurements of the signal $s=\Psi x$. The theory of compressed sensing guarantees that the sparsest solution (i.e., representation of the signal $s=\Psi x$ in terms of the basis $\Psi$) of $Ax=b$ can be obtained by solving \eqref{basis-pursuit} under certain conditions on the matrix $\Phi$ and the sparsity of $x$. This means that \eqref{basis-pursuit} gives the optimal solution of the following NP-hard problem \cite{Natarajan-95}:
\bea\label{ell-0-min}
\min_{x\in\br^n} \|x\|_0 \quad \st \quad  Ax=b,
\eea where $\|x\|_0$ counts the number of nonzero elements of $x$.

Matrix generalizations of  \eqref{ell-0-min} and \eqref{basis-pursuit}, respectively, are the so-called matrix rank minimization problem
\bea\label{matrix-rank-min} \min_{X\in\br^\mtn} \rank(X) \quad \st \quad  \ACal(X)= d, \eea
and its convex relaxation, the nuclear norm minimization problem:
\bea\label{nuclear-norm-min} \min_{X\in\br^\mtn} \|X\|_* \quad \st \quad  \ACal(X)= d, \eea
where $\ACal:\br^\mtn\rightarrow\br^p$ is a linear operator, $d\in\br^p$, and $\|X\|_*$ (the nuclear norm of $X$) is defined as the sum of singular values of matrix $X$. A special case of \eqref{matrix-rank-min} is the matrix completion problem:
\bea\label{matrix-rank-min-MC} \min_{X\in\br^\mtn} \rank(X) \quad \st \quad  X_{ij} = M_{ij}, \forall (i,j)\in\Omega, \eea
whose convex relaxation is given by:
\bea\label{nuclear-norm-min-MC} \min_{X\in\br^\mtn} \|X\|_* \quad \st \quad  X_{ij} = M_{ij}, \forall (i,j)\in\Omega. \eea
The matrix completion problem has a lot of interesting applications in online recommendation systems, collaborative filtering \cite{Srebro-thesis-2004,Srebro-Jaakkola-2003}, etc., including the
famous Netflix problem \cite{netflixprize}. It has been proved that under certain conditions, the solutions of the NP-hard problems \eqref{matrix-rank-min} and \eqref{matrix-rank-min-MC} are given respectively by solving their convex relaxations \eqref{nuclear-norm-min} and \eqref{nuclear-norm-min-MC}, with high probability (see, e.g., \cite{Recht-Fazel-Parrilo-2007,Candes-Recht-2008,Candes-Tao-2009,Recht-simpler-MC-2009,Gross-anybasis-MC-2010}).

The linearized Bregman (LB) method was proposed in \cite{YinOsherGoldfarbDarbon2008} to solve the basis pursuit problem \eqref{basis-pursuit}. The method was derived by linearizing the quadratic penalty term in the augmented Lagrangian function that is minimized on each iteration of the so-called Bregman method introduced in \cite{Osher-Burger-Goldfarb-Xu-Yin-05} while adding a prox term to it. The linearized Bregman method was further analyzed in \cite{Cai-Candes-Shen-2008,Cai-Osher-Shen-lbreg1-08, Generalization-LB-Wotao} and applied to solve the matrix completion problem \eqref{nuclear-norm-min-MC} in \cite{Cai-Candes-Shen-2008}.

Throughout of this paper, we will sometimes focus our analysis on the basis pursuit problem \eqref{basis-pursuit}. However, all of the analysis and results can be easily extended to \eqref{nuclear-norm-min} and \eqref{nuclear-norm-min-MC}. The linearized Bregman method depends on a single parameter $\mu>0$ and, as the analysis in \cite{Cai-Candes-Shen-2008,Cai-Osher-Shen-lbreg1-08} shows, actually solves the problem
\bea\label{basis-pursuit-regularized} \min_{x\in\br^n} g_\mu(x):=\|x\|_1+\frac{1}{2\mu}\|x\|_2^2, \quad \st \quad Ax = b, \eea
rather than the problem \eqref{basis-pursuit}. Recently it was shown in \cite{Generalization-LB-Wotao} that the solution to \eqref{basis-pursuit-regularized} is also a solution to problem \eqref{basis-pursuit} as long as $\mu$ is chosen large enough. Furthermore, it was shown in \cite{Generalization-LB-Wotao} that the linearized Bregman method can be viewed as a gradient descent method applied to the Lagrangian dual of problem \eqref{basis-pursuit-regularized}.
 %\bea\label{basis-pursuit-regularized-dual} \max_y \min_x \{J(x)+\frac{1}{2\mu}\|x\|^2-\langle y, Ax-b\rangle\} =\max_y\{ b^\top y - g^*_\mu(A^\top y)\}  \eea
% where $g_\mu^*(v) := \max_x\{\langle v,x\rangle - g_\mu(x)\}$ denotes the Fenchel dual (or convex conjugate) of $g_\mu(x)$.
This dual problem is an unconstrained optimization problem of the form
\be\label{prob:min-G-mu} \min_{y\in \br^m} \quad G_\mu(y), \ee
where the objective function $G_\mu(y)$ is differentiable since $g_\mu(x)$ is strictly convex (see, e.g., \cite{Rockafellar-book-70}).
Motivated by this result, some techniques for speeding up the classical gradient descent method applied to this dual problem such as taking Barzilai-Borwein (BB) steps \cite{Barzilai-Borwein-88}, and incorporating it into a limited memory BFGS (L-BFGS) method \cite{Liu-Nocedal-89}, were proposed in \cite{Generalization-LB-Wotao}. Numerical results on the basis pursuit problem \eqref{basis-pursuit} reported in \cite{Generalization-LB-Wotao} show that the performance of the linearized Bregman method can be greatly improved by using these techniques.

Our starting point is also motivated by the equivalence between applying the linearized Bregman method to \eqref{basis-pursuit} and solving the Lagrangian dual problem \eqref{prob:min-G-mu} by the gradient descent method. Since the gradient of $G_\mu(y)$ can be shown to be Lipschitz continuous, it is well-known that the classical gradient descent method with a properly chosen step size will obtain an $\epsilon$-optimal solution to \eqref{prob:min-G-mu} (i.e., an approximate solution $y^k$ such that $G_\mu(y^k)-G_\mu(y^*)\leq\epsilon$) in $O(1/\epsilon)$ iterations. In \cite{Nesterov-1983}, Nesterov proposed a technique for accelerating the gradient descent method for solving problem of the form \eqref{prob:min-G-mu} (see, also, \cite{NesterovConvexBook2004}), and proved that using this accelerated method, the number of iterations needed to obtain an $\epsilon$-optimal solution is reduced to $O(1/\sqrt{\epsilon})$ with a negligible change in the work required at each iteration. Nesterov also proved that the $O(1/\sqrt{\epsilon})$ complexity bound is the best bound that one can get if one uses only the first-order information. Based on the above discussion, we propose an accelerated linearized Bregman (ALB) method for solving \eqref{basis-pursuit-regularized} which is equivalent to an accelerated gradient descent method for solving the Lagrangian dual \eqref{prob:min-G-mu} of \eqref{basis-pursuit-regularized}. As a by-product, we show that the basic and the accelerated linearized Bregman methods require $O(1/\epsilon)$ and $O(1/\sqrt{\epsilon})$ iterations, respectively, to obtain an $\epsilon$-optimal solution with respect to the Lagrangian for \eqref{basis-pursuit-regularized}.

The rest of this paper is organized as follows. In Section \ref{sec:LB} we describe the original Bregman iterative method, as well as the linearized Bregman method. We motivate the methods and state some previously obtained theoretical results that establish the equivalence between the LB method and a gradient descent method for the dual of problem \eqref{basis-pursuit-regularized}. We present our accelerated linearized Bregman method in Section \ref{sec:ALB}.  We also provide a theoretical foundation for the accelerated algorithm and prove complexity results for it and the unaccelerated method. In Section \ref{sec:constraint-extension}, we describe how the LB and ALB methods can be extended to basis pursuit problems that include additional convex constraints. In Section \ref{sec:Numerical}, we report preliminary numerical results, on several compressed sensing basis pursuit and matrix completion problems. These numerical results show that our accelerated linearized Bregman method significantly outperforms the basic linearized Bregman method. We make some conclusions in Section \ref{sec:conclusion}.

\section{Bregman and Linearized Bregman Methods}\label{sec:LB}

The Bregman method was introduced to the image processing community by Osher \etal in \cite{Osher-Burger-Goldfarb-Xu-Yin-05} for solving the total-variation (TV) based image restoration problems. The Bregman distance \cite{Bregman-67} with respect to convex function $J(\cdot)$ between points $u$ and $v$ is defined as
\bea\label{Bregman-distance} D_J^p(u,v):=J(u)-J(v)-\langle p,u-v\rangle,\eea where $p\in\partial J(v)$, the subdifferential of $J$ at $v$. The Bregman method for solving \eqref{general-convex-min} is given below as Algorithm \ref{alg:Bregman}. Note that the updating formula for $p^k$ (Step 4 in Algorithm \ref{alg:Bregman}) is based on the optimality conditions of Step 3 in Algorithm \ref{alg:Bregman}:
 \beaa 0 \in \partial J(x^{k+1}) - p^k + A^\top(Ax^{k+1}-b).\eeaa
 This leads to \beaa p^{k+1} = p^k - A^\top(Ax^{k+1}-b). \eeaa It was shown in \cite{Osher-Burger-Goldfarb-Xu-Yin-05,YinOsherGoldfarbDarbon2008} that the Bregman method (Algorithm \ref{alg:Bregman}) converges to a solution of \eqref{general-convex-min} in a finite number of steps.

\begin{algorithm}[h]\caption{Original Bregman Iterative Method}\label{alg:Bregman}
\begin{algorithmic}[1]
\STATE { Input:} $x^0=p^0=0$.
\FOR{$k=0,1,\cdots$}
\STATE  $x^{k+1} = \arg\min_x D_J^{p^k}(x, x^k) + \frac{1}{2}\|Ax-b\|^2$;
\STATE  $p^{k+1} = p^k - A^\top(Ax^{k+1}-b)$;
\ENDFOR
\end{algorithmic}
\end{algorithm}

It is worth noting that for solving \eqref{general-convex-min}, the Bregman method is equivalent to the augmented Lagrangian method \cite{Hestenes-69,Powell-72,Rockafellar1976} in the following sense.

\begin{theorem}\label{the:equivalence-Bregman-AugLag}
The sequences $\{x^k\}$ generated by Algorithm \ref{alg:Bregman} and by the augmented Lagrangian method, which computes for $k=0, 1, \cdots$
\bea\label{augmented-Lagrangian-method}\left\{\ba{lll} x^{k+1} & := & \arg\min_x J(x) - \langle\lambda^k,Ax-b\rangle + \half\|Ax-b\|^2 \\ \lambda^{k+1} & := & \lambda^k - (Ax^{k+1}-b)\ea\right. \eea
starting from $\lambda^0=0$ are exactly the same.
\end{theorem}
\begin{proof}
 From Step 4 of Algorithm \ref{alg:Bregman} and the fact that $p^0=0$, it follows that $p^{k}=-\sum_{j=1}^k A^\top(Ax^j-b)$. From the second equation in \eqref{augmented-Lagrangian-method} and using $\lambda^0=0$, we get $\lambda^k = -\sum_{j=1}^k (Ax^j-b)$. Thus, $p^k=A^\top \lambda^k$ for all $k$. Hence it is easy to see that Step 3 of Algorithm \ref{alg:Bregman} is exactly the same as the first equation in \eqref{augmented-Lagrangian-method} and that the $x^{k+1}$ computed in Algorithm 1 and \eqref{augmented-Lagrangian-method} are exactly the same. Therefore, the sequences $\{x^k\}$ generated by both algorithms are exactly the same.
\end{proof}

\begin{comment}
\begin{eqnarray*}
x^{k+1} &=& arg\min_x J(x) - \langle p^k, x\rangle + \frac{1}{2}\|Ax-b\|^2 \\
p^{k+1} &=& p^k + A^T(b-Ax^{k+1}) \\
&=& \sum^k_{i=0}A^T(b-Ax^i)
\end{eqnarray*}

\noindent Therefore $x^{k+1}$ can be rewritten as
\begin{eqnarray*}
x^{k+1} &=& arg\min J(x) - \sum^{k-1}_{i=0} \langle b-Ax^i, Ax\rangle + \frac{1}{2}\|Ax-b\|^2
\end{eqnarray*}

\noindent Let $\lambda^k = \sum^{k-1}_{i=0} (b-Ax^i)$, we have
\begin{eqnarray*}
p^k = A^T\lambda^k
\end{eqnarray*}

In [], it is shown that the Bregman method is equivalent to the Augmented Lagrangian Method (ALM), i.e.,
\begin{eqnarray*}
x^{k+1} &=& arg\min_x J(x) - \langle \lambda^k, Ax-b\rangle + \frac{1}{2}\|Ax-b\|^2 \\
\lambda^{k+1} &=& \lambda^k - (Ax-b)
\end{eqnarray*}
\end{comment}

Note that for $J(x):=\alpha\|x\|_1$, Step 3 of Algorithm \ref{alg:Bregman} reduces to an $\ell_1$-regularized problem:
\bea\label{step3-Bregman-L1}\min_x \quad \alpha\|x\|_1-\langle p^k,x \rangle + \half\|Ax-b\|^2. \eea Although there are many algorithms for solving the subproblem \eqref{step3-Bregman-L1} such as FPC \cite{Hale-Yin-Zhang-SIAM-2008}, SPGL1 \cite{vandenBerg-Friedlander-2008}, FISTA \cite{Beck-Teboulle-2009} etc., it often takes them many iterations to do so. The linearized Bregman method was proposed in \cite{YinOsherGoldfarbDarbon2008}, and used in \cite{Osher-Mao-Dong-Yin-08,Cai-Osher-Shen-lbreg1-08,Cai-Osher-Shen-lbreg2-09} to overcome this difficulty. The linearized Bregman method replaces the quadratic term $\half\|Ax-b\|^2$ in the objective function that is minimized in Step 3 of Algorithm \ref{alg:Bregman} by its linearization $\langle A^\top(Ax^k-b), x\rangle$ plus a proximal term $\frac{1}{2\mu}\|x-x^k\|^2$. Consequently the updating formula for $p^k$ is changed since the optimality conditions for this minimization step become: \beaa 0 \in\partial J(x^{k+1})-p^k + A^\top(Ax^{k}-b) + \frac{1}{\mu}(x^{k+1}-x^k).\eeaa

In Algorithm \ref{alg:linearized-Bregman} below we present a slightly generalized version of the original linearized Bregman method that includes an additional parameter $\tau$ that corresponds to the length of a gradient step in a dual problem.

\begin{algorithm}[h]\caption{Linearized Bregman Method}\label{alg:linearized-Bregman}
\begin{algorithmic}[1]
\STATE { Input:} $x^0=p^0=0$, $\mu > 0$ and $\tau > 0$.
\FOR{$k=0,1,\cdots$}
\STATE  $x^{k+1} = \arg\min_x D_J^{p^k}(x,x^k) + \tau\langle A^\top(Ax^k-b), x\rangle + \frac{1}{2\mu}\|x-x^k\|^2$;
\STATE  $p^{k+1} = p^k - \tau A^\top(Ax^k-b) - \frac{1}{\mu}(x^{k+1}-x^k)$;
\ENDFOR
\end{algorithmic}
\end{algorithm}
In \cite{Generalization-LB-Wotao}, it is shown that when $\mu \|A\|^2 < 2$, where $\|A\|$ denotes the largest singular value of $A$, the iterates of the linearized Bregman method (Algorithm 2 with $\tau=1$) converge to the solution of the following regularized version of problem \eqref{general-convex-min}:
\begin{eqnarray}\label{smoothed-pursuit}
\min_x \quad J(x) + \frac{1}{2\mu}\|x\|^2 \quad \st \quad Ax=b.
\end{eqnarray}
We prove in Theorem \ref{the:equivalence-linearized-Bregman-gradient} below an analogous result for Algorithm 2 for a range of values of $\tau$. However, we first prove, as in \cite{Generalization-LB-Wotao}, that the linearized Bregman method (Algorithm 2) is equivalent to a gradient descent method
\be\label{gradient-descent-dual-pursuit} y^{k+1} := y^k - \tau \nabla G_{\mu}(y^k)\ee applied to the Lagrangian dual
\begin{equation*}
\max_y\min_w\{J(w)+\frac{1}{2\mu}\|w\|^2-\langle y, Aw-b\rangle \}
\end{equation*}
of \eqref{smoothed-pursuit}, which we express as the following equivalent minimization problem:
\begin{equation}\label{dual-pursuit}\min_y\quad G_\mu(y) := -\{J(w^*)+\frac{1}{2\mu}\|w^*\|^2 - \langle y, Aw^*-b\rangle\},\end{equation}
where \[w^* := \arg\min_w\{J(w)+\frac{1}{2\mu}\|w\|^2 - \langle y, Aw-b\rangle \}.\]
To show that $G_\mu(y)$ is continuously differentiable, we rewrite $G_\mu(y)$ as \[G_\mu(y)=-\Phi_\mu(\mu A^\top y)+\frac{\mu}{2}\|A^\top y\|^2-b^\top y,\] where \[\Phi_\mu(v)\equiv \min_w\{J(w)+\frac{1}{2\mu}\|w-v\|^2\}\] is strictly convex and continuously differentiable with gradient $\nabla\Phi_\mu(v)=\frac{v-\hat{w}}{\mu}$, and $\hat{w}=\arg\min_w\{J(w)+\frac{1}{2\mu}\|w-v\|^2\}$ (e.g., see Proposition 4.1 in \cite{Bertsekas-Tsitsiklis-1989}). From this it follows that $\nabla G_\mu(y)=A w^*-b$. Hence the gradient method \eqref{gradient-descent-dual-pursuit} corresponds to Algorithm \ref{alg:linearized-Bregman-equivalent} below.

\begin{algorithm}[h]\caption{Linearized Bregman Method (Equivalent Form)}\label{alg:linearized-Bregman-equivalent}
\begin{algorithmic}[1]
\STATE { Input:} $\mu > 0$, $\tau > 0$ and $y^0 = \tau b$.
\FOR{$k=0,1,\cdots$}
\STATE  $w^{k+1} := \arg \min_w \{J(w) + \frac{1}{2\mu}\|w\|^2 - \langle y^k, Aw-b \rangle\}$;
\STATE  $y^{k+1} := y^k - \tau(Aw^{k+1}-b).$;
\ENDFOR
\end{algorithmic}
\end{algorithm}

Lemma \ref{lem:equivalence-linearized-Bregman-gradient} and Theorem \ref{the:equivalence-linearized-Bregman-gradient} below generalize Theorem 2.1 in \cite{Generalization-LB-Wotao} by allowing a step length choice in the gradient step \eqref{gradient-descent-dual-pursuit} and show that Algorithms \ref{alg:linearized-Bregman} and \ref{alg:linearized-Bregman-equivalent} are equivalent. Our proof closely follows the proof of Theorem 2.1 in \cite{Generalization-LB-Wotao}.

\begin{lemma}\label{lem:equivalence-linearized-Bregman-gradient}
$x^{k+1}$ computed by Algorithm \ref{alg:linearized-Bregman} equals $w^{k+1}$ computed by Algorithm \ref{alg:linearized-Bregman-equivalent} if and only if
\bea\label{equivalence-lb-gradient-induction}A^\top y^k = p^k-\tau A^\top(Ax^k-b)+\frac{1}{\mu}x^k. \eea
\end{lemma}
\begin{proof}
By comparing Step 3 in Algorithms \ref{alg:linearized-Bregman} and \ref{alg:linearized-Bregman-equivalent}, it is obvious that $w^{k+1}$ is equal to $x^{k+1}$ if and only if \eqref{equivalence-lb-gradient-induction} holds.
\end{proof}
\begin{theorem}\label{the:equivalence-linearized-Bregman-gradient}
The sequences $\{x^k\}$ and $\{w^k\}$ generated by Algorithms \ref{alg:linearized-Bregman} and \ref{alg:linearized-Bregman-equivalent} are the same.
\end{theorem}
\begin{proof}
We prove by induction that equation \eqref{equivalence-lb-gradient-induction} holds for all $k\geq 0$. Note that \eqref{equivalence-lb-gradient-induction} holds for $k=0$ since $p^0=x^0=0$ and $y^0=\tau b$. Now let us assume that \eqref{equivalence-lb-gradient-induction} holds for all $0\leq k\leq n-1$; thus by Lemma \ref{lem:equivalence-linearized-Bregman-gradient} $w^{k+1}=x^{k+1}$ for all $0\leq k\leq n-1$. By iterating Step 4 in Algorithm \ref{alg:linearized-Bregman-equivalent} we get \bea\label{iterate-y} y^n = y^{n-1} - \tau(A w^{n}-b) = -\sum_{j=0}^n \tau(A x^j-b).\eea By iterating Step 4 in Algorithm \ref{alg:linearized-Bregman} we get \beaa p^n = -\sum_{j=0}^{n-1}\tau A^\top(Ax^j-b)-\frac{1}{\mu}x^n, \eeaa which implies that
\beaa p^n - \tau A^\top (Ax^n-b) + \frac{1}{\mu} x^n = -\sum_{j=0}^k\tau A^\top(A x^j-b) = A^\top y^n, \eeaa where the last equality follows from \eqref{iterate-y}; thus by induction \eqref{equivalence-lb-gradient-induction} holds for all $k\geq 0$, which implies by Lemma \ref{lem:equivalence-linearized-Bregman-gradient} that $x^k=w^k$ for all $k\geq 0$.
\end{proof}

Before analyzing Algorithms \ref{alg:linearized-Bregman} and \ref{alg:linearized-Bregman-equivalent}, we note that by defining $v^k = A^\top y^k$ and algebraically manipulating the last two terms in the objective function in Step 3 in Algorithm \ref{alg:linearized-Bregman-equivalent}, Steps 3 and 4 in that algorithm can be replaced by
\bea\label{lb-equiv-formula}
\left\{ \ba{lll}
w^{k+1} & := & \arg\min_w J(w) + \frac{1}{2\mu}\|w-\mu v^k\|^2 \\
v^{k+1} & := & v^k - \tau A^\top(Aw^{k+1}-b)
\ea \right. \eea
if we set $v^0 = \tau A^\top b$.
\begin{comment}
\begin{corollary}
The linearized Bregman method (Algorithm \ref{alg:linearized-Bregman}) is equivalent to applying the gradient method to the following Lagrangian dual problem, i.e., \eqref{dual-pursuit}:
\beaa \left\{\ba{lll} x^{k+1} & := & \arg\min_x \{J(x) + \frac{1}{2\mu}\|x\|^2 -\langle y^k, A^\top x-b\rangle\} \\ y^{k+1} & := & y^k - \tau(Ax^{k+1}-b). \ea\right. \eeaa
\end{corollary}
\begin{proof}
This result follows since Algorithm \ref{alg:linearized-Bregman} and the gradient descent method \eqref{gradient-descent-dual-pursuit} are equivalent.
\end{proof}
\end{comment}
Because Algorithms \ref{alg:linearized-Bregman} and \ref{alg:linearized-Bregman-equivalent} are equivalent, convergence results for the gradient descent method can be applied to both of them. Thus we have the following convergence result.
\begin{theorem}\label{the:convergence-dual-gradient}
Let $J(w) \equiv \|w\|_1$. Then $G_\mu(y)$ in the dual problem \eqref{dual-pursuit} is continuously differentiable and its gradient is Lipschitz continuous with the Lipschitz constant $L \leq \mu\|A\|^2$. Consequently, if the step length $\tau < \frac{2}{\mu\|A\|^2}$, the sequences $\{x^k\}$ and $\{w^k\}$ generated by Algorithms \ref{alg:linearized-Bregman} and \ref{alg:linearized-Bregman-equivalent} converge to the optimal solution of \eqref{smoothed-pursuit}.
\end{theorem}
\begin{proof}
When $J(x) = \|x\|_1$, $w^{k+1}$ in \eqref{lb-equiv-formula} reduces to
\begin{equation*}
w^{k+1} = \mu \cdot \shrink(v^k,1),
\end{equation*}
where the $\ell_1$ shrinkage operator is defined as
\be \label{shrinkage} \shrink(z,\alpha) := \sgn(z)\circ \max\{|z|-\alpha,0\}, \forall z\in\br^n, \alpha>0.\ee
$G_\mu(y)$ is continuously differentiable since $g_\mu(x)$ is strictly convex. Since for any point $y$, $\nabla G_\mu(y) = Aw-b$, where $w=\mu \cdot\shrink(A^\top y, 1)$, it follows from the fact that the shrinkage operator is non-expansive, i.e.,
\begin{equation*}
\|\shrink(s, \alpha)-\shrink (t, \alpha)\| \leq \|s-t\|,\ \forall s, t, \alpha
\end{equation*}
that
\begin{align*}
\|\nabla G_\mu(y^1) - \nabla G_\mu(y^2)\| & = \|\mu\cdot A \cdot \shrink(A^\top y^1,1) - \mu \cdot A\cdot\shrink(A^\top y^2,1)\| \\
                                          & \leq \mu \cdot \|A\| \cdot \|A^\top (y^1-y^2)\| \\
                                          & \leq \mu \|A\|^2 \|y^1-y^2\|,
\end{align*}
for any two points $y^1$ and $y^2$. Thus the Lipschitz constant $L$ of $\nabla G_\mu(\cdot)$ is bounded above by $\mu\|A\|^2$.

When $\tau < \frac{2}{\mu\|A\|^2}$, we have $\tau L < 2 $ and thus $|1-\tau L| < 1$. It then follows that the gradient
descent method $y^{k+1} = y^k - \tau \nabla G_\mu(y^k)$ converges and therefore Algorithms \ref{alg:linearized-Bregman} and \ref{alg:linearized-Bregman-equivalent} converge to $x_\mu^*$, the optimal solution of \eqref{smoothed-pursuit}.
\end{proof}

Before developing an accelerated version of the LB algorithm in the next section. We would like to comment on the similarities and differences between the LB method and Nesterov's composite gradient method \cite{Nesterov-07} and the ISTA method \cite{Beck-Teboulle-2009} applied to problem \eqref{general-convex-min} and related problems. The latter algorithms iterate Step 3 in the LB method (Algorithm \ref{alg:linearized-Bregman}) with $p^k=0$, and never compute or update the subgradient vector $p^k$. More importantly, their methods solve the unconstrained problem
\begin{equation*}
\min_{x\in \br^n} \quad \|x\|_1 + \frac{1}{2\mu}\|Ax-b\|^2.
\end{equation*}
Hence, while these methods and the LB method both linearize the quadratic term $\|Ax-b\|^2$ while handling the nonsmooth term $\|x\|_1$ directly, they are very different.

Similar remarks apply to the accelerated LB method presented in the next section and fast versions of ISTA and Nesterov's composite gradient method.

\section{The Accelerated Linearized Bregman Algorithm}\label{sec:ALB}

Based on Theorem \ref{the:equivalence-linearized-Bregman-gradient}, i.e., the equivalence between the linearized Bregman method and the gradient descent method, we can accelerate the linearized Bregman method by techniques used to accelerate the classical gradient descent method. In \cite{Generalization-LB-Wotao}, Yin considered several techniques such as line search, BB step and L-BFGS, to accelerate the linearized Bregman method. Here we consider the acceleration technique proposed by Nesterov in \cite{Nesterov-1983,NesterovConvexBook2004}. This technique accelerates the classical gradient descent method in the sense that it reduces the iteration complexity significantly without increasing the per-iteration computational effort. For the unconstrained minimization problem \eqref{prob:min-G-mu}, Nesterov's accelerated gradient method replaces the gradient descent method \eqref{gradient-descent-dual-pursuit}
by the following iterative scheme:
\bea\label{unconstrained-min-accelerated-gradient-descent} \left\{\ba{lll} x^{k+1} & := & y^{k} - \tau \nabla G_\mu(y^{k}) \\
                                                                           y^{k+1}   & := & \alpha_k x^{k+1} + (1-\alpha_k) x^{k}, \ea\right. \eea
where the scalars $\alpha_k$ are specially chosen weighting parameters. A typical choice for $\alpha_k$ is $\alpha_k=\frac{3}{k+2}$. If $\tau$ is chosen so that $\tau\leq 1/L$, where $L$ is the Lipschitz constant for $\nabla G_\mu(\cdot)$, Nesterov's accelerated gradient method \eqref{unconstrained-min-accelerated-gradient-descent} obtains an $\epsilon$-optimal solution of \eqref{prob:min-G-mu} in $O(1/\sqrt{\epsilon})$ iterations, while the classical gradient method \eqref{gradient-descent-dual-pursuit} takes $O(1/\epsilon)$ iterations. Moreover, the per-iteration complexities of \eqref{gradient-descent-dual-pursuit} and \eqref{unconstrained-min-accelerated-gradient-descent} are almost the same since computing the gradient $\nabla G_\mu(\cdot)$ usually dominates the computational cost in each iteration.
Nesterov's acceleration technique has been studied and extended by many others for nonsmooth minimization problems and variational inequalities, e.g., see
\cite{Nesterov-2005,Nesterov-07,Beck-Teboulle-2009,Tseng-2008,Nemirovski-Prox-siopt-2005,Goldfarb-Ma-Ksplit,Goldfarb-Ma-Scheinberg-2010,Goldfarb-Scheinberg-fastlinesearch2011}.
%In the following, we use Nesterov's acceleration technique to construct an accelerated linearized Bregman method whose per-iteration computational complexity is almost the same as that for Algorithm \ref{alg:linearized-Bregman}, but whose iteration complexity is improved significantly.

Our accelerated linearized Bregman method is given below as Algorithm \ref{alg:ALB}. The main difference between it and the basic linearized Bregman method (Algorithm  \ref{alg:linearized-Bregman}) is that the latter uses the previous iterate $x^k$ and subgradient $p^k$ to compute the new iterate $x^{k+1}$, while Algorithm \ref{alg:ALB} uses extrapolations $\tilde{x}^k$ and $\tilde{p}^k$ that are computed as linear combinations of the two previous iterates and subgradients, respectively. Carefully choosing the sequence of weighting parameters $\{\alpha_k\}$ guarantees an improved rate of convergence.

\begin{algorithm}[h]\caption{Accelerated Linearized Bregman Method}\label{alg:ALB}
\begin{algorithmic}[1]
\STATE { Input:} $x^0=\tilde{x}^0=\tilde{p}^0 = p^0 = 0$, $\mu > 0$, $\tau >0$.
\FOR{$k=0,1,\cdots$}
\STATE  $x^{k+1} = \arg\min_x D_J^{\tilde{p}^k}(x,\tilde{x}^k) + \tau\langle A^\top(A\tilde{x}^k-b), x\rangle + \frac{1}{2\mu}\|x-\tilde{x}^k\|^2$;
\STATE  $p^{k+1} = \tilde{p}^k - \tau A^\top(A\tilde{x}^k-b) - \frac{1}{\mu}(x^{k+1}-\tilde{x}^k)$;
%\STATE  $\alpha^k = 1+\frac{t_k-1}{t_{k+1}}$;
%\STATE  $t_{k+1} = \frac{1+\sqrt(1+4t_k^2)}{2}$;
\STATE  $\tilde{x}^{k+1} = \alpha_k x^{k+1} + (1-\alpha_k) x^k$;
\STATE  $\tilde{p}^{k+1} = \alpha_k p^{k+1} + (1-\alpha_k) p^{k}$.
\ENDFOR
\end{algorithmic}
\end{algorithm}

In the following, we first establish the equivalence between the accelerated linearized Bregman method and the corresponding accelerated gradient descent method \eqref{unconstrained-min-accelerated-gradient-descent}, which we give explicitly as \eqref{fast-dual-GD} below applied to the dual problem \eqref{dual-pursuit}. Based on this, we then present complexity results for both basic and accelerated linearized Bregman methods. Not surprisingly, the accelerated linearized Bregman method improves the iteration complexity from $O(1/\epsilon)$ to $O(1/\sqrt{\epsilon})$.

\begin{theorem}\label{the:equiv-ALB-ADG}
The accelerated linearized Bregman method (Algorithm \ref{alg:ALB}) is equivalent to the accelerated dual gradient descent method \eqref{fast-dual-GD} starting from $\tilde{y}^0 = y^{0} = \tau b$:
\bea\label{fast-dual-GD}
\left\{\ba{lll}
w^{k+1}      & := & \arg\min J(w) + \frac{1}{2\mu}\|w\|^2 - \langle \tilde{y}^k, Aw-b\rangle  \\
y^{k+1}  & := & \tilde{y}^k - \tau(Aw^{k+1}-b) \\
\tilde{y}^{k+1}      & := & \alpha_k y^{k+1} + (1-\alpha_k) y^{k}.
\ea \right. \eea
More specifically, the sequence $\{x^k\}$ generated by Algorithm \ref{alg:ALB} is exactly the same as the sequence $\{w^k\}$ generated by \eqref{fast-dual-GD}.
\end{theorem}

\begin{proof}
Note that the Step 3 of Algorithm \ref{alg:ALB} is equivalent to
\bea\label{step3-ALB} x^{k+1} := \arg\min J(x) - \langle \tilde{p}^k,x\rangle + \tau\langle A^\top(A\tilde{x}^k-b),x\rangle+\frac{1}{2\mu}\|x-\tilde{x}^k\|^2. \eea
Comparing \eqref{step3-ALB} with the first equation in \eqref{fast-dual-GD}, it is easy to see that $x^{k+1}=w^{k+1}$ if and only if
\bea\label{equiv-ALB-ADG-induction-eq} A^\top \tilde{y}^k = \tilde{p}^k + \tau A^\top(b-A\tilde{x}^k)+\frac{1}{\mu}\tilde{x}^k.\eea
We will prove \eqref{equiv-ALB-ADG-induction-eq} in the following by induction. Note that \eqref{equiv-ALB-ADG-induction-eq} holds for $k=0$ since $\tilde{y}^0=\tau b$ and $\tilde{x}^0=\tilde{p}^0=0$. As a result, we have $x^1=w^1$. By defining $w^0=0$, we also have $x^0=w^0$,
\bea\label{equiv-ALB-ADG-induction-eq-k1-1} A^\top \tilde{y}^1 = A^\top(\alpha_0 y^1 + (1-\alpha_0) A^\top y^0) = \alpha_0 A^\top \tilde{y}^0  + \alpha_0\tau A^\top(b-Aw^1) + (1-\alpha_0) A^\top y^{0}. \eea
On the other hand,
\bea\label{equiv-ALB-ADG-induction-eq-k1-2} p^1 = \tilde{p}^0 + \tau A^\top(b-A\tilde{x}^0)-\frac{1}{\mu}(x^1-\tilde{x}^0) = A^\top \tilde{y}^0 - \frac{1}{\mu}x^1, \eea
where for the second equality we used \eqref{equiv-ALB-ADG-induction-eq} for $k=0$. Expressing $\tilde{p}^1$ and $\tilde{x}^1$ in terms of their affine combinations of $p^1$, $p^0$, $x^1$ and $x^0$, then substituting for $p^1$ using \eqref{equiv-ALB-ADG-induction-eq-k1-2} and using the fact that $x^0=p^0=0$, and finally using $\tilde{y}^0 = \tau b$ and \eqref{equiv-ALB-ADG-induction-eq-k1-1}, we obtain,
\begin{align*}
\tilde{p}^1+\tau A^\top(b-A\tilde{x}^1)+\frac{1}{\mu}\tilde{x}^1 & = \alpha_0 p^1 + (1-\alpha_0) p^0 + \alpha_0\tau A^\top(b-A x^1)+(1-\alpha_0)\tau A^\top(b - Ax^0)+\frac{1}{\mu}(\alpha_0 x^1+(1-\alpha_0)x^0) \\
                                                                 & = \alpha_0 (A^\top \tilde{y}^0-\frac{1}{\mu}x^1) + \alpha_0\tau A^\top(b-A x^1)+(1-\alpha_0)\tau A^\top b+\frac{1}{\mu}\alpha_0 x^1 \\
                                                                 & = \alpha_0 A^\top \tilde{y}^0 + \alpha_0\tau A^\top (b-Ax^1) + (1-\alpha_0)A^\top y^0 \\
                                                                 & = \alpha_0 A^\top \tilde{y}^0 + \alpha_0\tau A^\top (b-Aw^1) + (1-\alpha_0)A^\top y^0 \\
                                                                 & = A^\top \tilde{y}^1.
\end{align*}
Thus we proved that \eqref{equiv-ALB-ADG-induction-eq} holds for $k=1$. Now let us assume that \eqref{equiv-ALB-ADG-induction-eq} holds for $0 \leq k\leq n-1$, which implies $x^k=w^k, \forall 0 \leq k\leq n$ since $x^0=w^0$. We will prove that \eqref{equiv-ALB-ADG-induction-eq} holds for $k=n$.

First, note that
\bea\label{equiv-ALB-ADG-induction-eq-kn-1} p^n = \tilde{p}^{n-1}+\tau A^\top(b-A\tilde{x}^{n-1})-\frac{1}{\mu}(x^n-\tilde{x}^{n-1}) = A^\top \tilde{y}^{n-1} -\frac{1}{\mu}x^n, \eea where the first equality is from Step 4 of Algorithm \ref{alg:ALB} and the second equality is from \eqref{equiv-ALB-ADG-induction-eq} for $k=n-1$. From Step 6 of Algorithm \ref{alg:ALB} and \eqref{equiv-ALB-ADG-induction-eq-kn-1}, we have
\bea\label{equiv-ALB-ADG-induction-eq-kn-2}\ba{lll}\tilde{p}^n & = & \alpha_{n-1}p^n + (1-\alpha_{n-1})p^{n-1} \\
                                                             & = & \alpha_{n-1}(A^\top \tilde{y}^{n-1}-\frac{1}{\mu}x^n) + (1-\alpha_{n-1})(A^\top \tilde{y}^{n-2}-\frac{1}{\mu}x^{n-1}) \\
                                                             & = & \alpha_{n-1}A^\top \tilde{y}^{n-1} + (1-\alpha_{n-1})A^\top \tilde{y}^{n-2} - \frac{1}{\mu} \tilde{x}^n, \ea\eea
where the last equality uses Step 5 of Algorithm \ref{alg:ALB}. On the other hand, from \eqref{fast-dual-GD} we have
\bea\label{equiv-ALB-ADG-induction-eq-kn-3}\ba{lll}A^\top \tilde{y}^n & = & A^\top (\alpha_{n-1}y^{n}+(1-\alpha_{n-1})y^{n-1}) \\
                                                              & = & \alpha_{n-1}A^\top (\tilde{y}^{n-1}+\tau(b-Aw^n)) + (1-\alpha_{n-1})A^\top (\tilde{y}^{n-2}+\tau(b-Aw^{n-1})) \\
                                                              & = & \alpha_{n-1}A^\top \tilde{y}^{n-1} + (1-\alpha_{n-1})A^\top \tilde{y}^{n-2} + \tau A^\top [b-A(\alpha_{n-1}x^n+(1-\alpha_{n-1})x^{n-1})] \\
                                                              & = & \alpha_{n-1}A^\top \tilde{y}^{n-1} + (1-\alpha_{n-1})A^\top \tilde{y}^{n-2} + \tau A^\top(b-A\tilde{x}^n), \ea\eea
where the third equality is from $w^n=x^n$ and $w^{n-1}=x^{n-1}$, the last equality is from Step 5 of Algorithm \ref{alg:ALB}.
Combining \eqref{equiv-ALB-ADG-induction-eq-kn-2} and \eqref{equiv-ALB-ADG-induction-eq-kn-3} we get that \eqref{equiv-ALB-ADG-induction-eq} holds for $k=n$.
\end{proof}

Like the linearized Bregman, we can also use a simpler implementation for accelerated linearized Bregman method in which the main computation at each step is a proximal minimization. Specifically,
\eqref{fast-dual-GD} is equivalent to the following three steps.
\bea \label{alg:nice ALB}
\left\{\ba{lll}
w^{k+1}  & := & \arg\min J(w) + \frac{1}{2\mu}\|w-\mu \tilde{v}^k\|^2  \\
v^{k+1}  & := & \tilde{v}^k - \tau A^\top(Aw^{k+1}-b) \\
\tilde{v}^{k+1}      & := & \alpha_k v^{k+1} + (1-\alpha_k) v^{k}
\ea \right. \eea
As before this follows from letting $v^{k} = A^\top y^{k}$ and $\tilde{v}^k = A^\top \tilde{y}^k$ and completing the square in the objective function in the first equation of \eqref{fast-dual-GD}.

Next we prove iteration complexity bounds for both basic and accelerated linearized Bregman algorithms. Since these algorithms are standard gradient descent methods applied to the Lagrangian dual function and these results have been well established, our proofs will be quite brief.
\begin{theorem}\label{the:complexity-basic-LB}
Let the sequence $\{x^k\}$ be generated by the linearized Bregman method (Algorithm \ref{alg:linearized-Bregman}) and $(x^*, y^*)$ be the pair of optimal primal and dual solutions for Problem \eqref{smoothed-pursuit}. Let $\{y^k\}$ be the sequence generated by Algorithm \ref{alg:linearized-Bregman-equivalent} and suppose the step length $\tau \leq \frac{1}{L}$, where $L$ is the Lipschitz constant for $\nabla G_\mu(y)$. Then for the Lagrangian function
\begin{eqnarray}\label{G-L-1}
\mathcal{L}_{\mu}(x,y) = J(x) + \frac{1}{2\mu}\|x\|^2 - \langle y, Ax-b\rangle,
\end{eqnarray} we have
\begin{eqnarray}\label{complexity-basic-LB-conclusion}
\mathcal{L}_{\mu}(x^*, y^*) - \mathcal{L}_{\mu}(x^{k+1},y^k)   \leq \frac{\|y^*-y^0\|^2}{2\tau k}.
\end{eqnarray}
Thus, if we further have $\tau\geq \beta/L$, where $0< \beta\leq 1$, then $(x^{k+1},y^k)$ is an $\epsilon$-optimal solution to Problem \eqref{smoothed-pursuit} with respect to the Lagrangian function if $k\geq \lceil C/\epsilon \rceil$, where $C:=\frac{L\|y^*-y^0\|^2}{2\beta}$.
\end{theorem}
\begin{proof}
From \eqref{dual-pursuit} we get
\begin{eqnarray}\label{G-L-2}
G_\mu(y^k) = - \mathcal{L}_{\mu}(x^{k+1},y^k).
\end{eqnarray}

By using the convexity of function $G_\mu(\cdot)$ and the Lipschitz continuity of the gradient $\nabla G_\mu(\cdot)$, we get for any $y$,
%When $\tau < \frac{1}{\mu \|A\|^2}$, we get
%\begin{equation}
%\frac{1}{2\tau} \geq \frac{\mu\|A\|^2}{2} \geq \frac{L}{2}. \label{thm1:1}
%\end{equation}
\bea \label{complexity-basic-LB-long-eq} \ba{lll}
G_\mu(y^k) - G_\mu(y) & \leq & G_\mu(y^{k-1}) + \langle \nabla G_\mu(y^{k-1}), y^k-y^{k-1}\rangle + \frac{L}{2}\|y^k - y^{k-1}\|^2 - G_\mu(y) \\
                      & \leq & G_\mu(y^{k-1}) + \langle \nabla G_\mu(y^{k-1}), y^k-y^{k-1}\rangle + \frac{1}{2\tau}\|y^k - y^{k-1}\|^2 - G_\mu(y) \\
                      & \leq & \langle \nabla G_\mu(y^{k-1}), y^{k-1}-y \rangle + \langle \nabla G_\mu(y^{k-1}), y^k-y^{k-1}\rangle + \frac{1}{2\tau}\|y^k - y^{k-1}\|^2 \\
                      & = & \langle \nabla G_\mu(y^{k-1}), y^k-y\rangle + \frac{1}{2\tau}\|y^k - y^{k-1}\|^2 \\
                      & = & \frac{1}{\tau}\langle y^{k-1}-y^k, y^{k}-y\rangle + \frac{1}{2\tau}\|y^k - y^{k-1}\|^2 \\
                      & \leq & \frac{1}{2\tau}(\|y-y^{k-1}\|^2 - \|y-y^{k}\|^2).
\ea\eea
Setting $y=y^{k-1}$ in \eqref{complexity-basic-LB-long-eq}, we obtain $G_\mu(y^k)\leq G_\mu(y^{k-1})$ and thus the sequence $\{G_\mu(y^k)\}$ is non-increasing. Moreover, summing \eqref{complexity-basic-LB-long-eq} over $k=1,2,\ldots,n$ with $y=y^*$ yields
\begin{align*}
n(G_\mu(y^n)-G_\mu(y^*)) \leq \sum_{k=1}^n (G_\mu(y^k)-G_\mu(y^*)) \leq \frac{1}{2\tau}(\|y^*-y^0\|^2-\|y^*-y^n\|^2) \leq \frac{1}{2\tau}\|y^*-y^0\|^2,
\end{align*}
and this implies \eqref{complexity-basic-LB-conclusion}.
\end{proof}

Before we analyze the iteration complexity of the accelerated linearized Bregman method, we introduce a lemma from \cite{Tseng-2008} that we will use in our analysis.
\begin{lemma}[Property 1 in \cite{Tseng-2008}]\label{lem:property-1-Tseng}
For any proper lower semicontinuous function $\psi: \br^n \rightarrow (-\infty, +\infty]$ and any $z\in\br^n$, if
\[z_+ = \arg\min_x \{\psi(x)+\half\|x-z\|^2\},\] then \[\psi(x) + \half\|x-z\|^2\geq \psi(z_+)+\half\|z_+-z\|^2+\half\|x-z_+\|^2, \quad \forall x\in\br^n.\]
\end{lemma}

The following theorem gives an iteration-complexity result for the accelerated linearized Bregman method. Our proof of this theorem closely follows the proof of Proposition 2 in \cite{Tseng-2008}.
\begin{theorem}\label{complexity-ALB}
Let the sequence $\{x^k\}$ be generated by accelerated linearized Bregman method (Algorithm \ref{alg:ALB}) and $(x^*, y^*)$ be the optimal primal and dual variable for Problem \eqref{smoothed-pursuit}. Let $\{\alpha_k\}$ be chosen as \bea\label{update-alpha}\alpha_{k-1} &=& 1 + \theta_k(\theta_{k-1}^{-1}-1),\eea where
\bea\label{update-t} \theta_{-1} := 1, \mbox{ and } \theta_k = \frac{2}{k+2}, \forall k\geq 0. \eea
Let the sequence $\{y^k\}$ be defined as in \eqref{fast-dual-GD} and the step length $\tau \leq \frac{1}{L}$, where $L$ is the Lipschitz constant of $\nabla G_\mu(y)$ and $G_\mu(\cdot)$ is defined by \eqref{G-L-2}. We have
\begin{eqnarray}\label{complexity-ALB-conclusion-inequa}
G_\mu(y^{k}) - G_\mu(y^*)  \leq \frac{2\|y^*-y^0\|^2}{\tau k^2}.
\end{eqnarray}
Thus, if we further have $\tau\geq \beta/L$, where $0< \beta\leq 1$, then $(x^{k+1},y^k)$ is an $\epsilon$-optimal solution to Problem \eqref{smoothed-pursuit} with respect to the Lagrangian function \eqref{G-L-1} if $k\geq \lceil \sqrt{C/\epsilon} \rceil$, where $C:=\frac{2L\|y^*-y^0\|^2}{\beta}$.
\end{theorem}
\begin{proof}
Let
\begin{equation}\label{thm3:1}
z^k = y^{k-1} + \theta_{k-1}^{-1}(y^k-y^{k-1})
\end{equation}
and denote the linearization of $G_\mu(y)$ as
\begin{equation}\label{thm3:2}
l_{G_\mu}(x;y) := G_\mu(y) + \langle \nabla G_\mu(y), x-y\rangle \leq G_\mu(x).
\end{equation}
Therefore the second equality in \eqref{fast-dual-GD} is equivalent to
\begin{eqnarray*}
y^{k+1} &:=& \arg \min_y \quad G_\mu(\tilde{y}^k) + \langle \nabla G_\mu(\tilde{y}^k), y-\tilde{y}^k \rangle + \frac{1}{2\tau}\|y-\tilde{y}^k\|^2 \\
&=& \arg \min_y \quad l_{G_\mu}(y;\tilde{y}^k) + \frac{1}{2\tau}\|y-\tilde{y}^k\|^2.
\end{eqnarray*}
Define $\hat{y}^k :=(1-\theta_k)y^k + \theta_k y^*$, we have
\begin{align}\label{proof-complexity-ALB-long-eq}
G_\mu(y^{k+1}) & \leq G_\mu(\tilde{y}^k) + \langle \nabla G_\mu(\tilde{y}^k), y^{k+1}-\tilde{y}^k \rangle + \frac{L}{2}\|y^{k+1}-\tilde{y}^k\|^2 \\ \nonumber
               & \leq l_{G_\mu}(y^{k+1};\tilde{y}^k) + \frac{1}{2\tau}\|y^{k+1}-\tilde{y}^k\|^2 \\\nonumber
               & \leq l_{G_\mu}(\hat{y}^k;\tilde{y}^k) + \frac{1}{2\tau}\|\hat{y}^k-\tilde{y}^k\|^2 - \frac{1}{2\tau}\|\hat{y}^k-y^{k+1}\|^2 \\\nonumber
               & = l_{G_\mu}((1-\theta_k)y^k + \theta_k y^*; \tilde{y}^k) + \frac{1}{2\tau}\|(1-\theta_k)y^k + \theta_k y^*-\tilde{y}^k\|^2 - \frac{1}{2\tau}\|(1-\theta_k)y^k + \theta_k y^*-y^{k+1}\|^2 \\\nonumber
               & = l_{G_\mu}((1-\theta_k)y^k + \theta_k y^*; \tilde{y}^k) + \frac{\theta_k^2}{2\tau}\|y^*+\theta_k^{-1}(y^k-\tilde{y}^k)-y^k\|^2 - \frac{\theta_k^2}{2\tau}\|y^* + \theta_k^{-1}(y^k-y^{k+1})-y^k\|^2 \\\nonumber
               & = l_{G_\mu}((1-\theta_k)y^k + \theta_k y^*; \tilde{y}^k) + \frac{\theta_k^2}{2\tau}\|y^*-z^k\|^2 - \frac{\theta_k^2}{2\tau}\|y^*-z^{k+1}\|^2 \\\nonumber
               & = (1-\theta_k)l_{G_\mu}(y^k;\tilde{y}^k) + \theta_k l_{G_\mu}(y^*;\tilde{y}^k) + \frac{\theta_k^2}{2\tau}\|y^*-z^k\|^2 - \frac{\theta_k^2}{2\tau}\|y^*-z^{k+1}\|^2 \\\nonumber
               & \leq (1-\theta_k)G_\mu(y^k) + \theta_k G_\mu(y^*) + \frac{\theta_k^2}{2\tau}\|y^*-z^k\|^2 - \frac{\theta_k^2}{2\tau}\|y^*-z^{k+1}\|^2,
\end{align}
where the second inequality is from \eqref{thm3:2} and $\tau\leq 1/L$, the third inequality uses Lemma \ref{lem:property-1-Tseng} with $\psi(x):=\tau l_{G_\mu}(x;\tilde{y}^k)$, the third equality uses \eqref{thm3:1}, \eqref{fast-dual-GD} and \eqref{update-alpha} and the last inequality uses \eqref{thm3:2}.

Therefore we get
\begin{eqnarray*}
\frac{1}{\theta_k^2}(G_\mu(y^{k+1})-G_\mu(y^*)) \leq \frac{1-\theta_k}{\theta_k^2}(G_\mu(y^k)-G_\mu(y^*)) + \frac{1}{2\tau}\|y-z^k\|^2 - \frac{1}{2\tau}\|y-z^{k+1}\|^2.
\end{eqnarray*}
From \eqref{update-t}, it is easy to show that $\frac{1-\theta_k}{\theta_k^2} \leq \frac{1}{\theta_{k-1}^2}$ for all $k\geq 0$. Thus \eqref{proof-complexity-ALB-long-eq} implies that
\begin{equation}\label{proof-complexity-ALB-last-eq}
\frac{1-\theta_{k+1}}{\theta_{k+1}^2}(G_\mu(y^{k+1})-G_\mu(y^*))\leq \frac{1-\theta_k}{\theta_{k}^2}(G_\mu(y^k)-G_\mu(y^*)) + \frac{1}{2\tau}\|y-z^k\|^2 - \frac{1}{2\tau}\|y-z^{k+1}\|^2.
\end{equation}
Summing \eqref{proof-complexity-ALB-last-eq} over $k=0,1,\ldots,n-1$, we get
\begin{equation*}
\frac{1-\theta_n}{\theta_{n}^2}(G_\mu(y^{n})-G_\mu(y^*)) \leq \frac{1}{2\tau}\|y^*-z^0\|^2 = \frac{1}{2\tau}\|y^*-y^0\|^2,
\end{equation*}
which immediately implies \eqref{complexity-ALB-conclusion-inequa}.
\end{proof}

\begin{remark}
The proof technique and the choice of $\theta_k$ used here are suggested in \cite{Tseng-2008} for accelerating the basic algorithm. Other choices of $\theta_k$ can be found in \cite{Nesterov-1983,NesterovConvexBook2004,Beck-Teboulle-2009,Tseng-2008}. They all work here and give the same order of iteration complexity.
\end{remark}

\section{Extension to Problems with Additional Convex Constraints}\label{sec:constraint-extension}
We now consider extensions of both the LB and ALB methods to problems of the form
\begin{eqnarray}\label{problem-convex-constraints}
\min_{x\in X} \quad J(x) \quad \textrm{s.t}\ \ Ax=b,
\end{eqnarray}
where $X$ is a nonempty closed convex set in $\br^n$. It is not clear how to extend the LB and ALB methods (Algorithms \ref{alg:linearized-Bregman} and \ref{alg:ALB}) to problem \eqref{problem-convex-constraints} since we can no longer rely on the relationship
\beaa 0 \in\partial J(x^{k+1})-p^k + A^\top(Ax^{k}-b) + \frac{1}{\mu}(x^{k+1}-x^k)\eeaa
to compute a subgradient $p^{k+1} \in \partial J(x^{k+1})$. Fortunately, the Lagrangian dual gradient versions of these algorithms do not suffer from this difficulty. All that is required to extend them to problem \eqref{problem-convex-constraints} is to include the constraint $w \in X$ in the minimization step in these algorithms. Note that the gradient of
\begin{eqnarray*}
\hat{\Phi}_\mu(v) = \min_{w\in X}\{J(w) + \frac{1}{2\mu}\|w-v\|^2\}
\end{eqnarray*}
remains the same. Also it is clear that the iteration complexity results given in Theorems \ref{the:complexity-basic-LB} and \ref{complexity-ALB} apply to these algorithms as well.

Being able to apply the LB and ALB methods to problems of the form of \eqref{problem-convex-constraints} greatly expands their usefulness. One immediate extension is to compressed sensing problems in which the signal is required to have nonnegative components. Also \eqref{problem-convex-constraints} directly includes all linear programs. Applying the LB and ALB to such problems, with the goal of only obtaining approximated optimal solutions, will be the subject of a future paper.

\section{Numerical Experiments}\label{sec:Numerical}
In this section, we report some numerical results that demonstrate the effectiveness of the accelerated linearized Bregman algorithm. All numerical experiments were run in MATLAB 7.3.0 on a Dell Precision 670 workstation with an Intel Xeon(TM) 3.4GHZ CPU and 6GB of RAM.

\subsection{Numerical Results on Compressed Sensing Problems}\label{sec:numerical-cs}
In this subsection, we compare the performance of the accelerated linearized Bregman method against the performance of the basic linearized Bregman method on a variety of compressed sensing problems of the form \eqref{basis-pursuit}.

We use three types of sensing matrices $A\in\br^\mtn$. Type (i): $A$ is a standard Gaussian matrix generated by the $randn(m,n)$ function in MATLAB. Type (ii): $A$ is first generated as a standard Gaussian matrix and then normalized to have unit-norm columns. Type (iii): The elements of $A$ are sampled from a Bernoulli distribution as either $+1$ or $-1$. We use two types of sparse solutions $x^*\in\br^n$ with sparsity $s$ (i.e., the number of nonzeros in $x^*$). The positions of the nonzero entries of $x^*$ are selected uniformly at random, and each nonzero value is sampled either from (i) standard Gaussian (the $randn$ function in MATLAB) or from (ii) $[-1,1]$ uniformly at random ($2*rand-1$ in MATLAB).

For compressed sensing problems, where $J(x)=\|x\|_1$, the linearized Bregman method reduces to the two-line algorithm:
\beaa\left\{\ba{lll}
x^{k+1} & := & \mu\cdot\shrink(v^k,1) \\
v^{k+1} & := & v^k + \tau A^\top(b-Ax^{k+1}),
\ea\right.\eeaa
where the $\ell_1$ shrinkage operator is defined in \eqref{shrinkage}.
Similarly, the accelerated linearized Bregman can be written as:
\beaa\left\{\ba{lll}
x^{k+1} & := & \mu\cdot\shrink(\tilde{v}^k,1) \\
v^{k+1}  & := & \tilde{v}^k + \tau A^T(b-Ax^{k+1}) \\
\tilde{v}^{k+1}      & := & \alpha_k v^{k+1} + (1-\alpha_k) v^{k}.
\ea\right.\eeaa Both algorithms are very simple to program and involve only one $Ax$ and one $A^\top y$ matrix-vector multiplication in each iteration.

We ran both LB and ALB with the $seed$ used for generating random number in MATLAB setting as $0$. Here we set $n=2000, m=0.4\times n, s = 0.2\times m, \mu = 5$ for all data sets. We set $\tau = \frac{2}{\mu \|A\|^2}$. We terminated the algorithms when the stopping criterion \begin{equation}\label{stop-crit} \|Ax^k-b\|/\|b\| < 10^{-5}\end{equation} was satisfied or the number of iterations exceeded 5000. Note that \eqref{stop-crit} was also used in \cite{Generalization-LB-Wotao}. We report the results in Table \ref{tab:CS-LB-ALB}.

\begin{table}[ht]\label{tab:CS-LB-ALB}
\begin{center}{ \caption{Compare linearized Bregman (LB) with accelerated linearized Bregman (ALB)}
\begin{tabular}{c|c|c|c|c|c}\hline
\multicolumn{2}{c|}{Standard Gaussian matrix $A$} & \multicolumn{2}{c|}{ Number of Iterations } & \multicolumn{2}{c}{ Relative error $\|x-x^*\|/\|x^*\|$}
\\\hline

Type of $x^*$  & $n (m=0.4n, s=0.2m)$ & LB & ALB & LB & ALB
\\\hline

Gaussian & 2000 & 5000+ & 330 & 5.1715e-3 & 1.4646e-5
\\\hline

Uniform & 2000 & 1681 & 214 & 2.2042e-5 & 1.5241e-5
 \\\hline

\multicolumn{2}{c|}{Normalized Gaussian matrix $A$} & \multicolumn{2}{c|}{ Number of Iterations } & \multicolumn{2}{c}{ Relative error $\|x-x^*\|/\|x^*\|$}
\\\hline
Type of $x^*$  & $n (m=0.4n, s=0.2m)$ & LB & ALB & LB & ALB
\\\hline
Gaussian & 2000 & 2625 & 234 & 3.2366e-5  & 1.2664e-5 \\\hline
Uniform & 2000 & 5000+ & 292 & 1.2621e-2 & 1.5629e-5  \\\hline

\multicolumn{2}{c|}{Bernoulli +1/-1 matrix $A$} & \multicolumn{2}{c|}{ Number of Iterations } & \multicolumn{2}{c}{ Relative error $\|x-x^*\|/\|x^*\|$}
\\\hline
Type of $x^*$  & $n (m=0.4n, s=0.2m)$ & LB & ALB & LB & ALB
\\\hline
Gaussian & 2000 & 2314 & 222  & 4.2057e-5  & 1.0812e-5  \\\hline
Uniform  & 2000 & 5000+ & 304 & 1.6141e-2 & 1.5732e-5 \\\hline
\end{tabular}}
\end{center}
\end{table}

In Table \ref{tab:CS-LB-ALB}, we see that for three out of six problems, LB did not achieve the desired convergence criterion within 5000 iterations, while ALB satisfied this stopping criterion in less than 330 iterations on all six problems. To further demonstrate the significant improvement the ALB achieved over LB, we plot in Figures \ref{fig:LB-ALB1}, \ref{fig:LB-ALB2} and \ref{fig:LB-ALB3} the Euclidean norms of the residuals and the relative errors as a function of the iteration number that were obtained by LB and ALB applied to the same data sets. These figures also depict the non-monotonic behavior of the ALB method.

\begin{center}
\begin{figure}
  \hspace*{-0.1in}\includegraphics[width=0.5\textwidth]{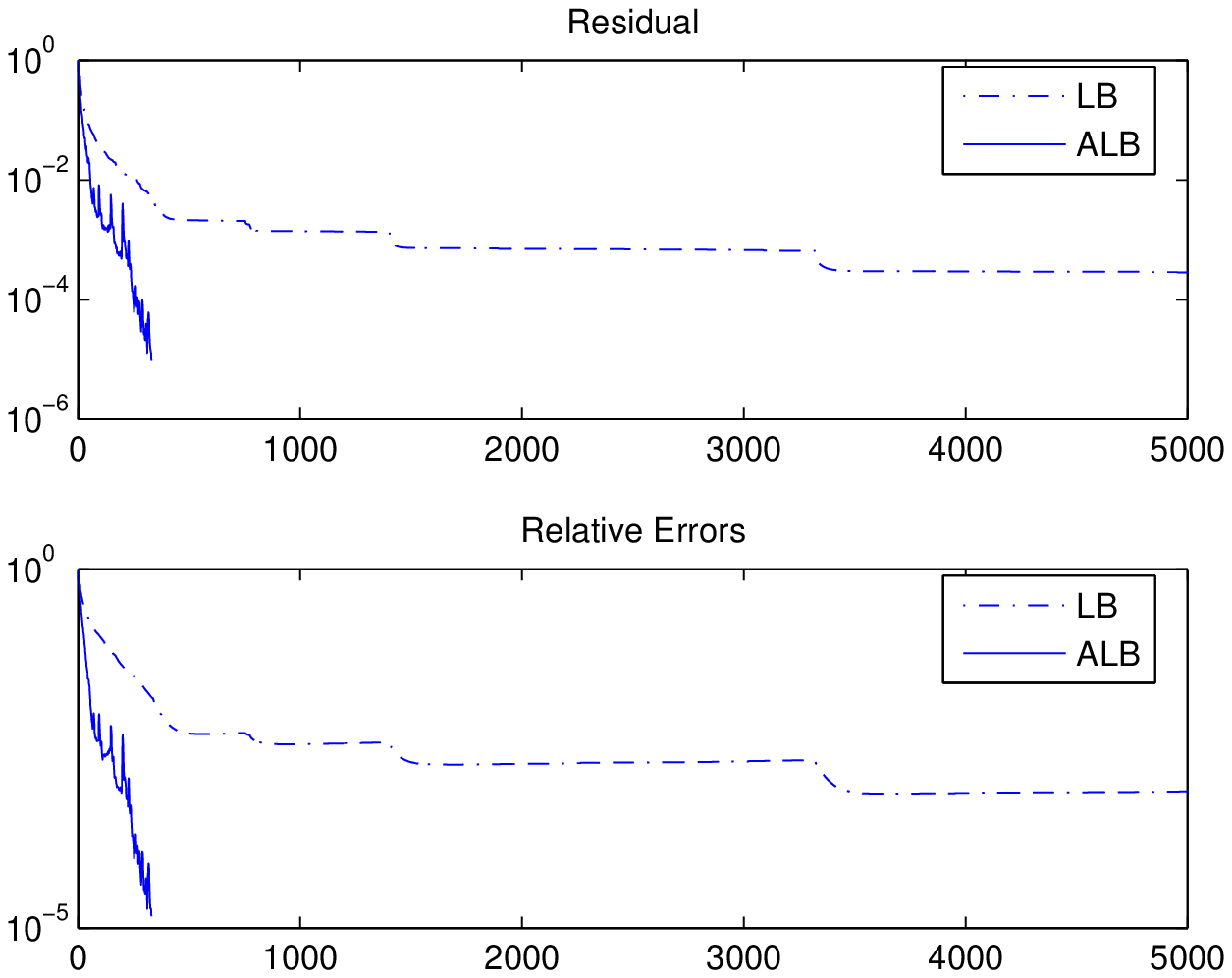}
  \hspace*{-0.1in}\includegraphics[width=0.5\textwidth]{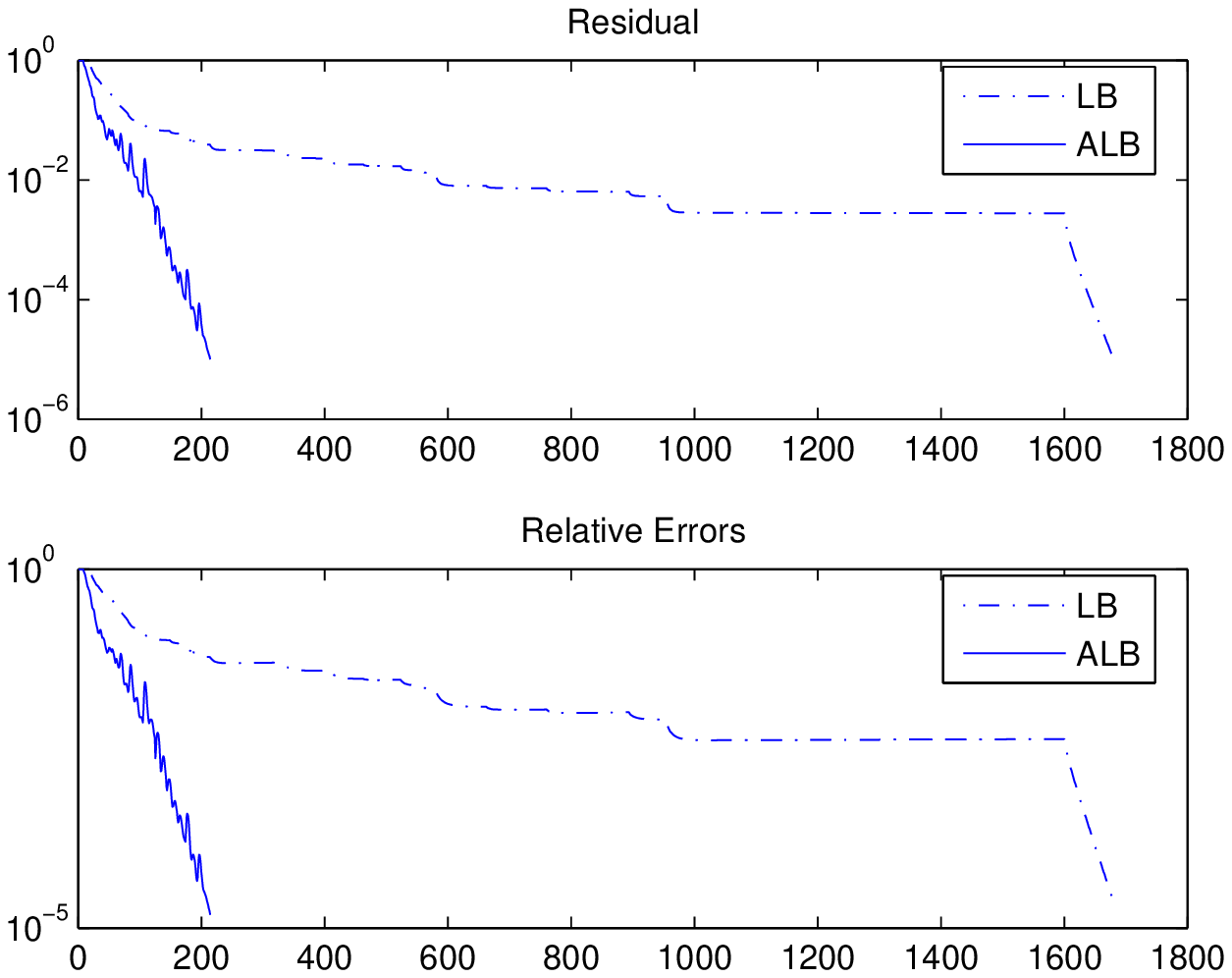}
  \caption{Gaussian matrix $A$, Left: Gaussian $x^*$, Right: Uniform $x^*$}\label{fig:LB-ALB1}
  \hspace*{-0.1in}\includegraphics[width=0.5\textwidth]{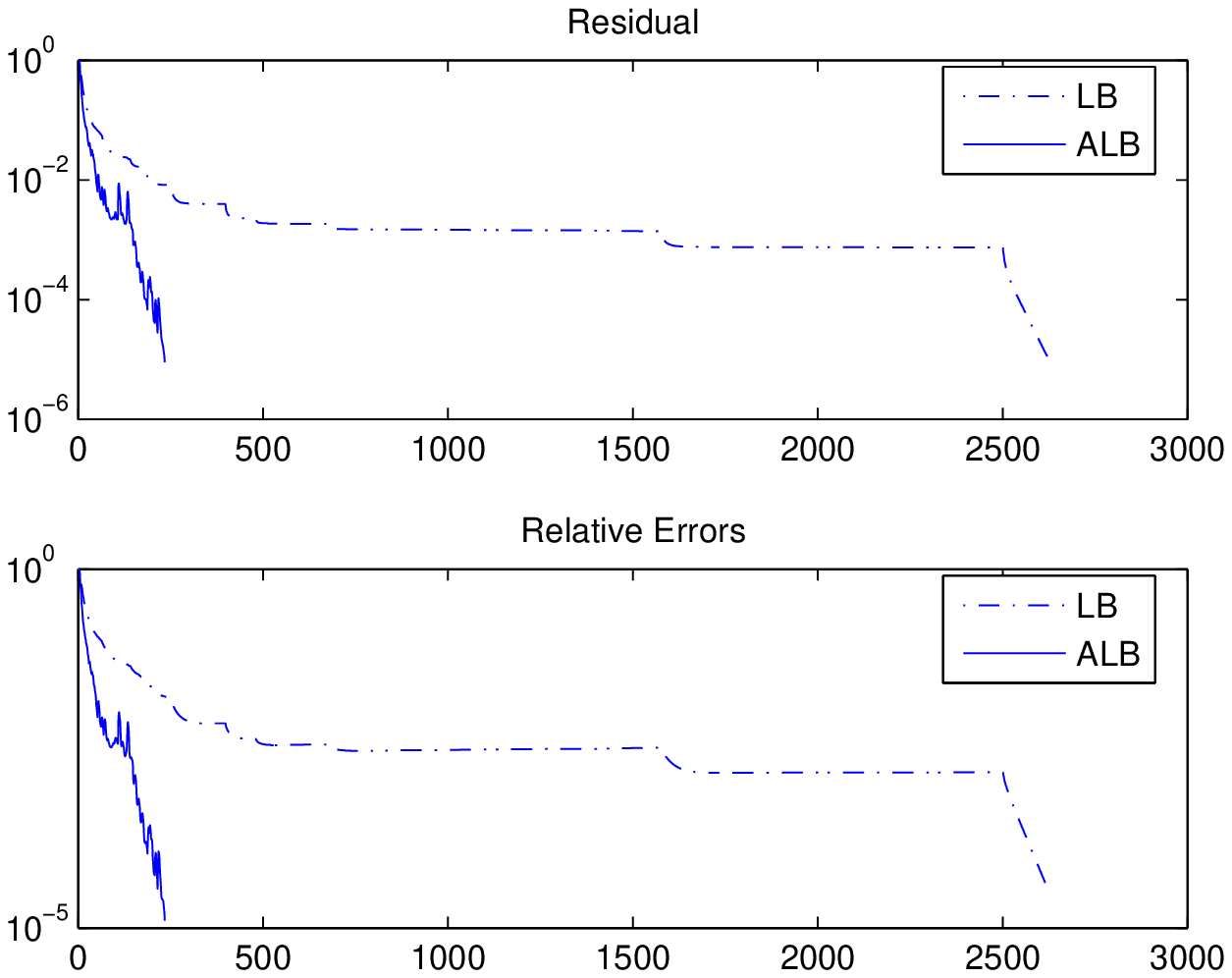}
  \hspace*{-0.1in}\includegraphics[width=0.5\textwidth]{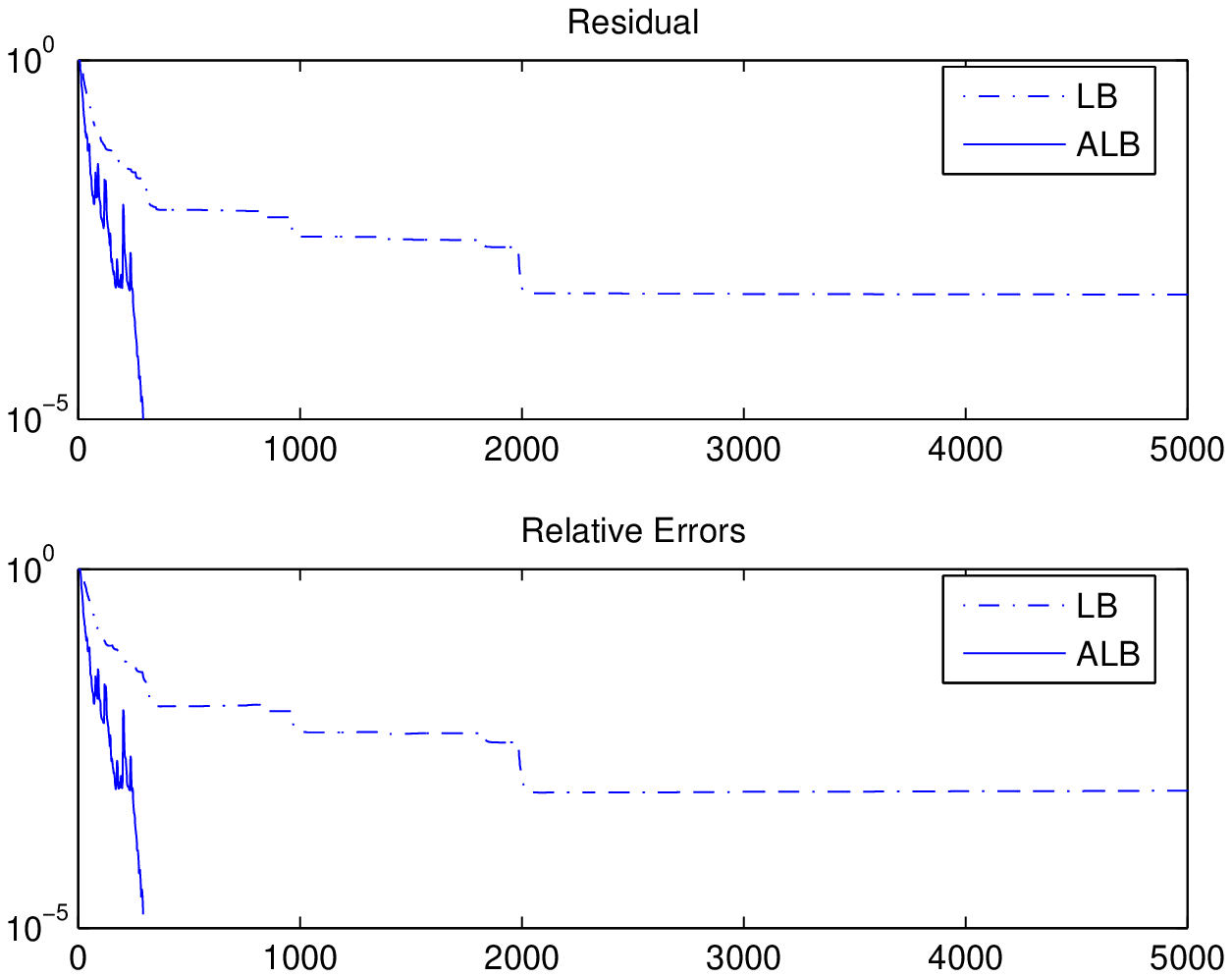}
  \caption{Normalized Gaussian matrix $A$, Left: Gaussian $x^*$, Right: Uniform $x^*$}\label{fig:LB-ALB2}
  \hspace*{-0.1in}\includegraphics[width=0.5\textwidth]{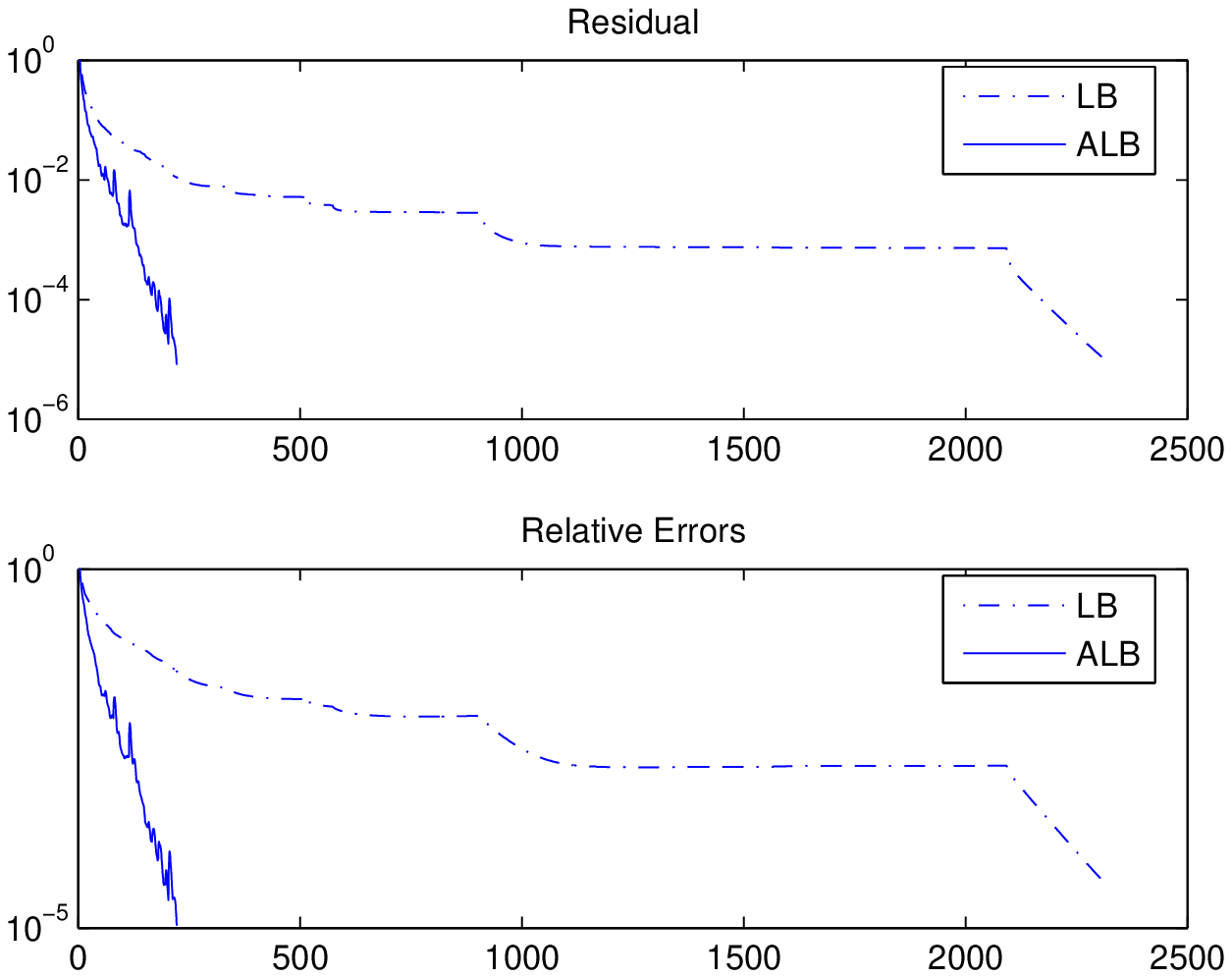}
  \hspace*{-0.1in}\includegraphics[width=0.5\textwidth]{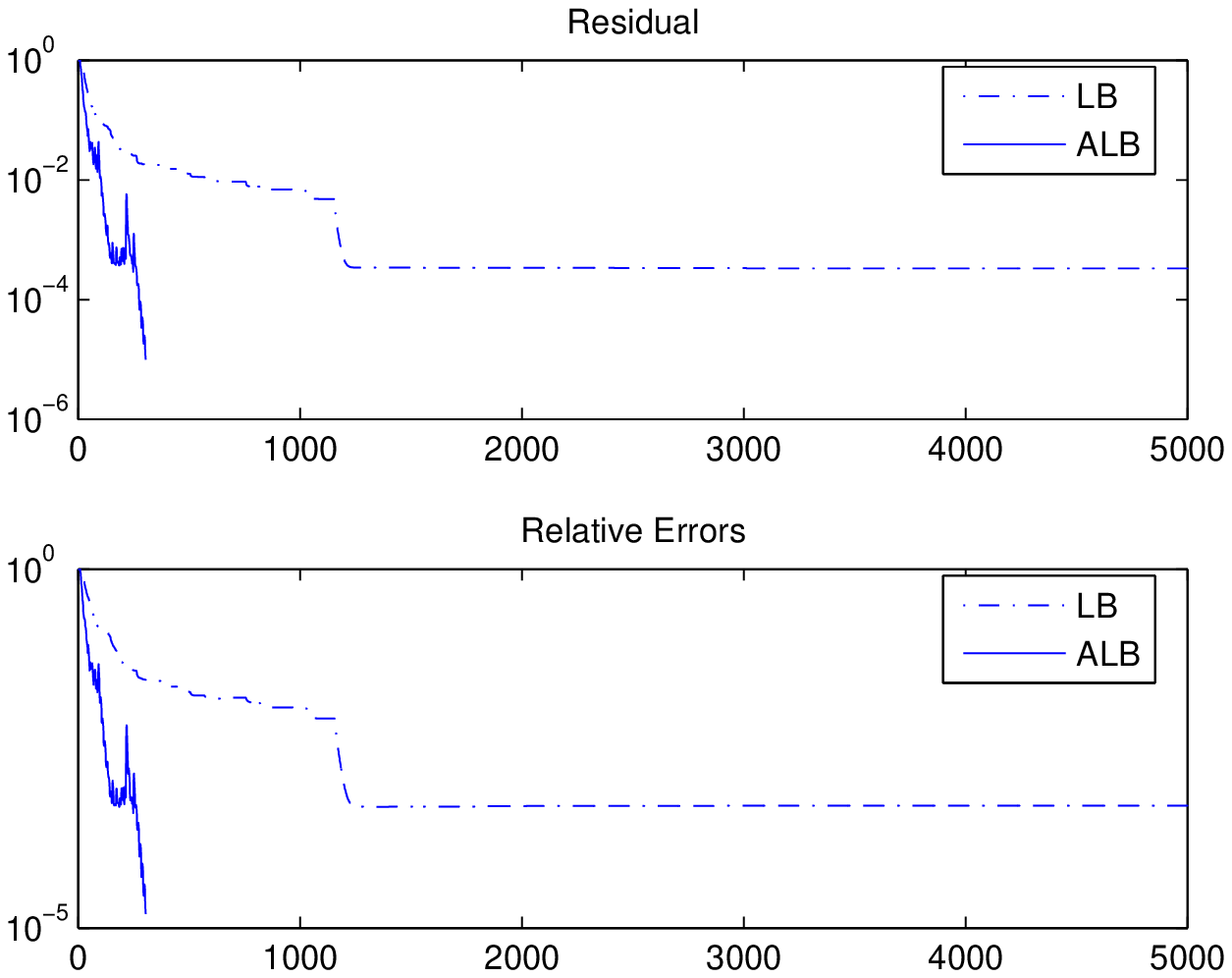}
  \caption{Bernoulli matrix $A$, Left: Gaussian $x^*$, Right: Uniform $x^*$}\label{fig:LB-ALB3}
\end{figure}
\end{center}

\subsection{Numerical Results on Matrix Completion Problems}\label{sec:numerical-mc}
There are fast implementations of linearized Bregman \cite{Cai-Candes-Shen-2008} and other solvers \cite{Ma-Goldfarb-Chen-2008,Toh-Yun-2009,Liu-Sun-Toh-2009,Wen-Yin-Zhang-MC-2010} for solving matrix completion problems. We do not compare the linearized Bregman and our accelerated linearized Bregman algorithms with these fast solvers here. Rather our tests are focused only on comparing ALB with LB and verifying that the acceleration actually occurs in practice for matrix completion problems.

The nuclear norm matrix completion problem \eqref{nuclear-norm-min-MC} can be rewritten as
\bea\label{nuclear-norm-min-MC-rewrite-oprator} \min_X \quad \|X\|_* \quad \st \quad \PCal_\Omega(X) = \PCal_\Omega(M), \eea where $[\PCal_\Omega(X)]_{ij}=X_{ij}$ if $(i,j)\in\Omega$ and $[\PCal_\Omega(X)]_{ij}=0$ otherwise. When the convex function $J(\cdot)$ is the nuclear norm of matrix $X$, the Step 3 of Algorithm \ref{alg:linearized-Bregman} with inputs $X^k,P^k$ can be reduced to
\bea\label{nuclear-norm-min-MC-rewrite-oprator-step3} X^{k+1}:= \arg\min_{X\in\br^\mtn}\mu\|X\|_*+\half\|X-(X^k-\mu(\tau \PCal_\Omega(\PCal_\Omega X^{k}-\PCal_\Omega(M))-P^k))\|_F^2.\eea It is known (see, e.g., \cite{Cai-Candes-Shen-2008,Ma-Goldfarb-Chen-2008}) that \eqref{nuclear-norm-min-MC-rewrite-oprator-step3} has the closed-form solution,
\beaa X^{k+1} = \Shrink(X^k-\mu(\tau \PCal_\Omega(\PCal_\Omega X^{k}-\PCal_\Omega(M))-P^k),\mu), \eeaa where the matrix shrinkage operator is defined as
\beaa \Shrink(Y,\gamma):= U \Diag(\max(\sigma-\gamma,0)) V^\top, \eeaa and $Y=U\Diag(\sigma)V^\top$ is the singular value decomposition (SVD) of matrix $Y$. Thus, a typical iteration of the linearized Bregman method (Algorithm \ref{alg:linearized-Bregman}), with initial inputs $X^0=P^0=0$, for solving the matrix completion problem \eqref{nuclear-norm-min-MC-rewrite-oprator} can be summarized as
\bea\label{alg:linearized-Bregman-MC}\left\{\ba{lll}X^{k+1} & := & \Shrink(X^k-\mu(\tau \PCal_\Omega(\PCal_\Omega X^{k}-\PCal_\Omega(M))-P^k),\mu) \\
                                                    P^{k+1} & := & P^k - \tau (\PCal_\Omega X^k - \PCal_\Omega M) - (X^{k+1}-X^k)/\mu. \ea\right. \eea
Similarly, a typical iteration of the accelerated linearized Bregman method (Algorithm \ref{alg:ALB}), with initial inputs $X^0=P^0=\tilde{X}^0=\tilde{P}^0=0$, for solving the matrix completion problem \eqref{nuclear-norm-min-MC-rewrite-oprator} can be summarized as
\bea\label{alg:ALB-MC}\left\{\ba{lll} X^{k+1} & := & \Shrink(X^k-\mu(\tau \PCal_\Omega(\PCal_\Omega X^{k}-\PCal_\Omega(M))-P^k),\mu) \\
                                      P^{k+1} & := & \tilde{P}^k - \tau (\PCal_\Omega\tilde{X}^k -\PCal_\Omega M) - (X^{k+1}-\tilde{X}^k)/\mu \\
                                      \tilde{X}^{k+1} & := & \alpha_k X^{k+1} + (1-\alpha_k) X^k \\
                                      \tilde{P}^{k+1} & := & \alpha_k P^{k+1} + (1-\alpha_k) P^k, \ea\right.\eea
where the sequence $\alpha_k$ is chosen according to Theorem \ref{complexity-ALB}.

We compare the performance of LB and ALB on a variety of matrix completion problems. We created matrices $M\in\br^{n\times n}$ with rank $r$ by the following procedure. We first created standard Gaussian matrices $M_L\in\br^{n\times r}$ and $M_R\in\br^{n\times r}$ and then we set $M=M_LM_R^\top$. The locations of the $p$ known entries in $M$ were sampled uniformly, and the values of these $p$ known entries were drawn from an iid Gaussian distribution. The ratio $p/n^2$ between the number of measurements and the number of entries in the matrix is denoted by ``SR'' (sampling ratio). The ratio between the dimension of the set of $n\times n$ rank $r$ matrices, $r(2n-r)$, and the number of samples $p$, is denoted by ``FR''.
In our tests, we fixed $FR$ to 0.2 and 0.3 and $r$ to $10$. We tested five matrices with dimension $n=100,200,300,400,500$ and set the number $p$ to $r(2n-r)/FR$. The random seed for generating random matrices in MATLAB was set to $0$. $\mu$ was set to $5n$ (a heuristic argument for this choice can be found in \cite{Cai-Candes-Shen-2008}). We set the step length $\tau$ to $1/\mu$ since for matrix completion problems $\|\PCal_\Omega\|=1$. We terminated the code when the relative error between the residual and the true matrix was less than $10^{-4}$, i.e.,
\bea\label{stop-crit-MC} \|\PCal_\Omega(X^k)-\PCal_\Omega(M)\|_F / \|\PCal_\Omega(M)\|_F < 10^{-4}. \eea Note that this stopping criterion was used in \cite{Cai-Candes-Shen-2008}. We also set the maximum number of iteration to 2000.

We report the number of iterations needed by LB and ALB to reach \eqref{stop-crit-MC} in Table \ref{tab:MC-LB-ALB}. Note that performing the shrinkage operation, i.e., computing an SVD, dominates the computational cost in each iteration of LB and ALB. Thus, the per-iteration complexities of LB and ALB are almost the same and it is reasonable to compare the number of iterations needed to reach the stopping criterion. We report the relative error $err:=\|X^k-M\|_F/\|M\|_F$ between the recovered matrix $X^k$ and the true matrix $M$ in Table \ref{tab:MC-LB-ALB}. We see from Table \ref{tab:MC-LB-ALB} that ALB needed significantly fewer iterations to meet the stopping criterion \eqref{stop-crit-MC}.

In Figures \ref{fig:MC-FR0p2} and \ref{fig:MC-FR0p3}, we plot the Frobenius norms of the residuals and the relative errors obtained by LB and ALB for iteration 1-500 for the tests involving matrices with dimension $n=200, 300, 400$ and $500$. Note that the non-monotonicity of ALB is far less pronounced on these problems.

\begin{table}[ht]\label{tab:MC-LB-ALB}
\begin{center}\caption{Comparison between LB and ALB on Matrix Completion Problems}
\begin{tabular}{c||c|c|c|c|c||c|c|c|c|c}\hline
    &  \multicolumn{5}{|c||}{$FR = 0.2, \rank = 10$} & \multicolumn{5}{|c}{$FR = 0.3, \rank = 10$} \\\hline
$n$ & SR & iter-LB & err-LB & iter-ALB & err-ALB & SR & iter-LB & err-LB & iter-ALB & err-ALB \\\hline

100 & 0.95 &  85 & 1.07e-4 & 63 & 1.11e-4 & 0.63 & 294 & 1.75e-4 & 163 & 1.65e-4 \\\hline

200 & 0.49 & 283 & 1.62e-4 & 171 & 1.58e-4 & 0.33 & 1224 & 3.76e-4 & 289 & 1.83e-4 \\\hline

300 & 0.33 & 466 & 1.64e-4 & 261 & 1.60e-4 & 0.22 & 2000+ & 3.59e-3 & 406 & 1.93e-4 \\\hline

400 & 0.25 & 667 & 1.79e-4 & 324 & 1.65e-4 & 0.17 & 2000+ & 1.12e-2 & 455 & 1.80e-4 \\\hline

500 & 0.20 & 831 & 1.76e-4 & 398 & 1.65e-4 & 0.13 & 2000+ & 3.14e-2 & 1016 & 7.49e-3 \\\hline
\end{tabular}
\end{center}
\end{table}

\begin{center}
\begin{figure}
  \hspace*{-0.1in}\includegraphics[width=0.5\textwidth]{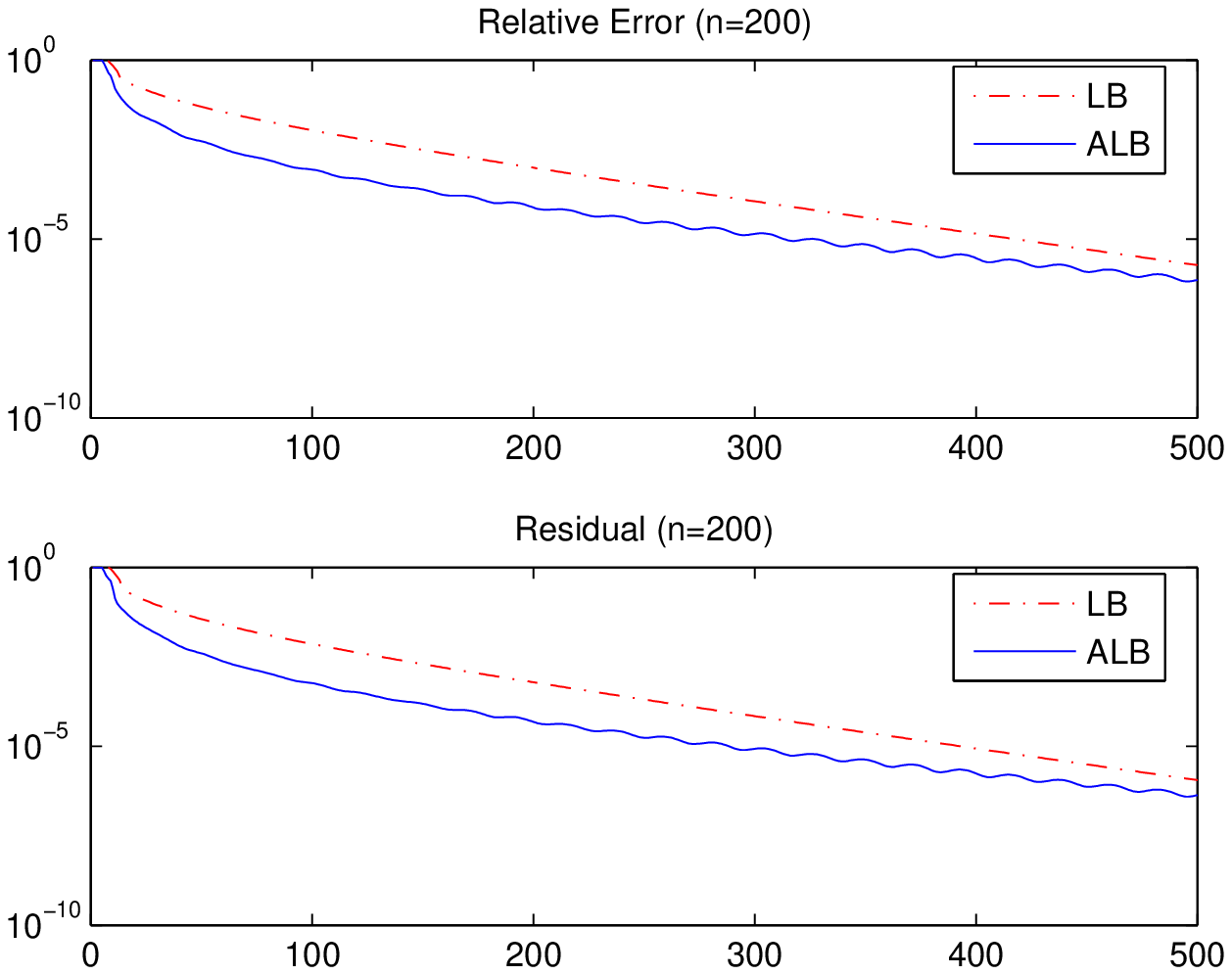}
  \hspace*{-0.1in}\includegraphics[width=0.5\textwidth]{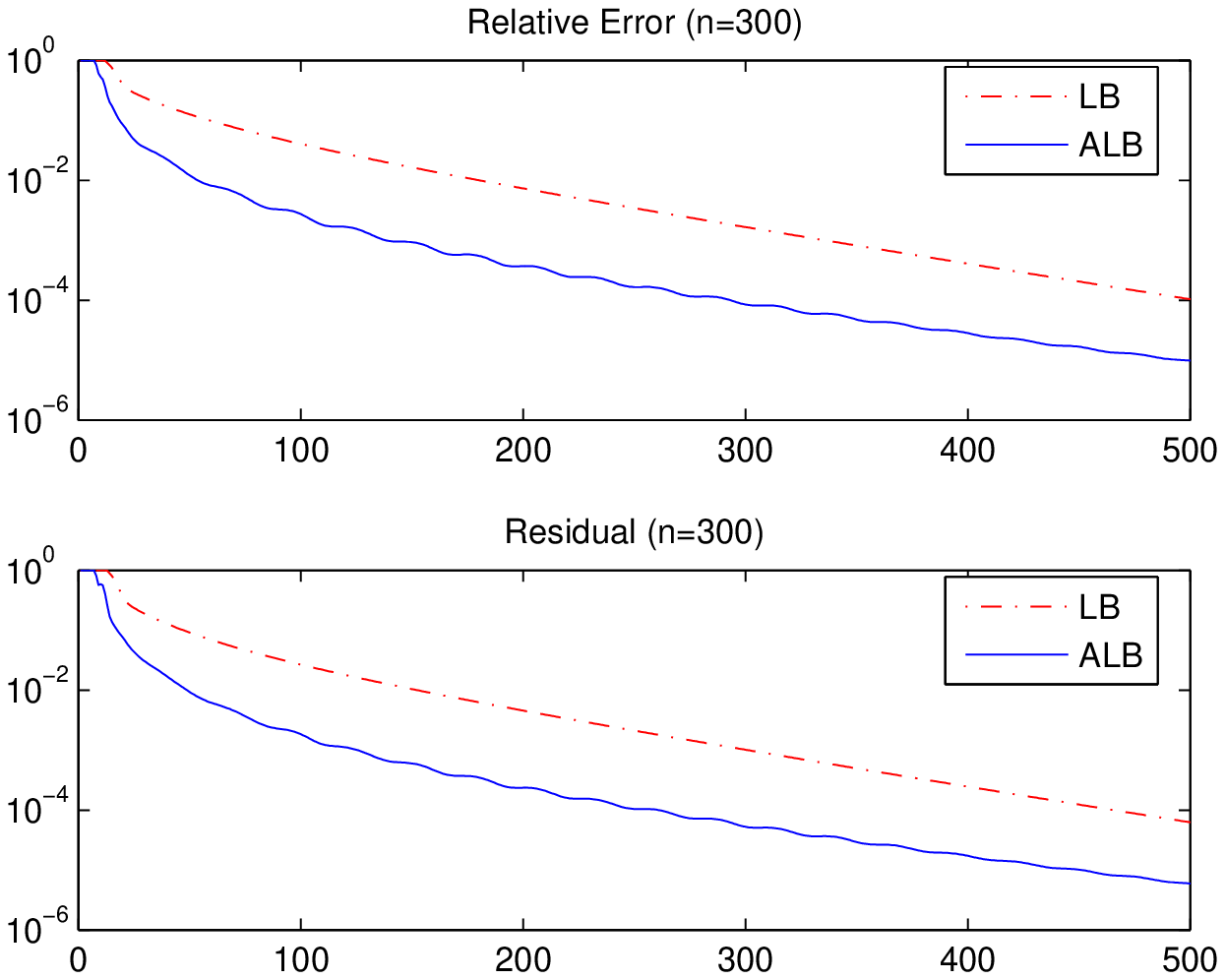}\\
  \hspace*{-0.1in}\includegraphics[width=0.5\textwidth]{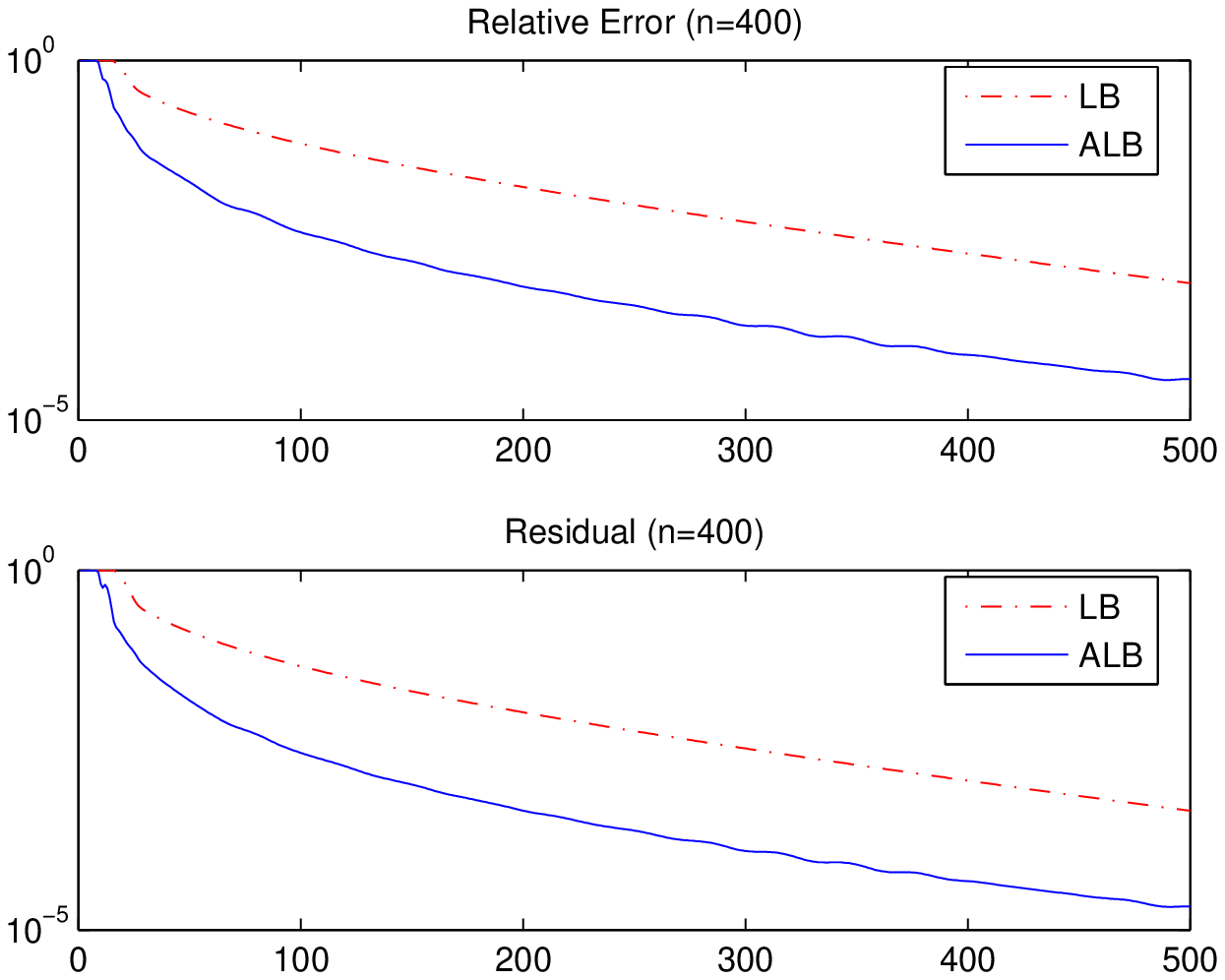}
  \hspace*{-0.1in}\includegraphics[width=0.5\textwidth]{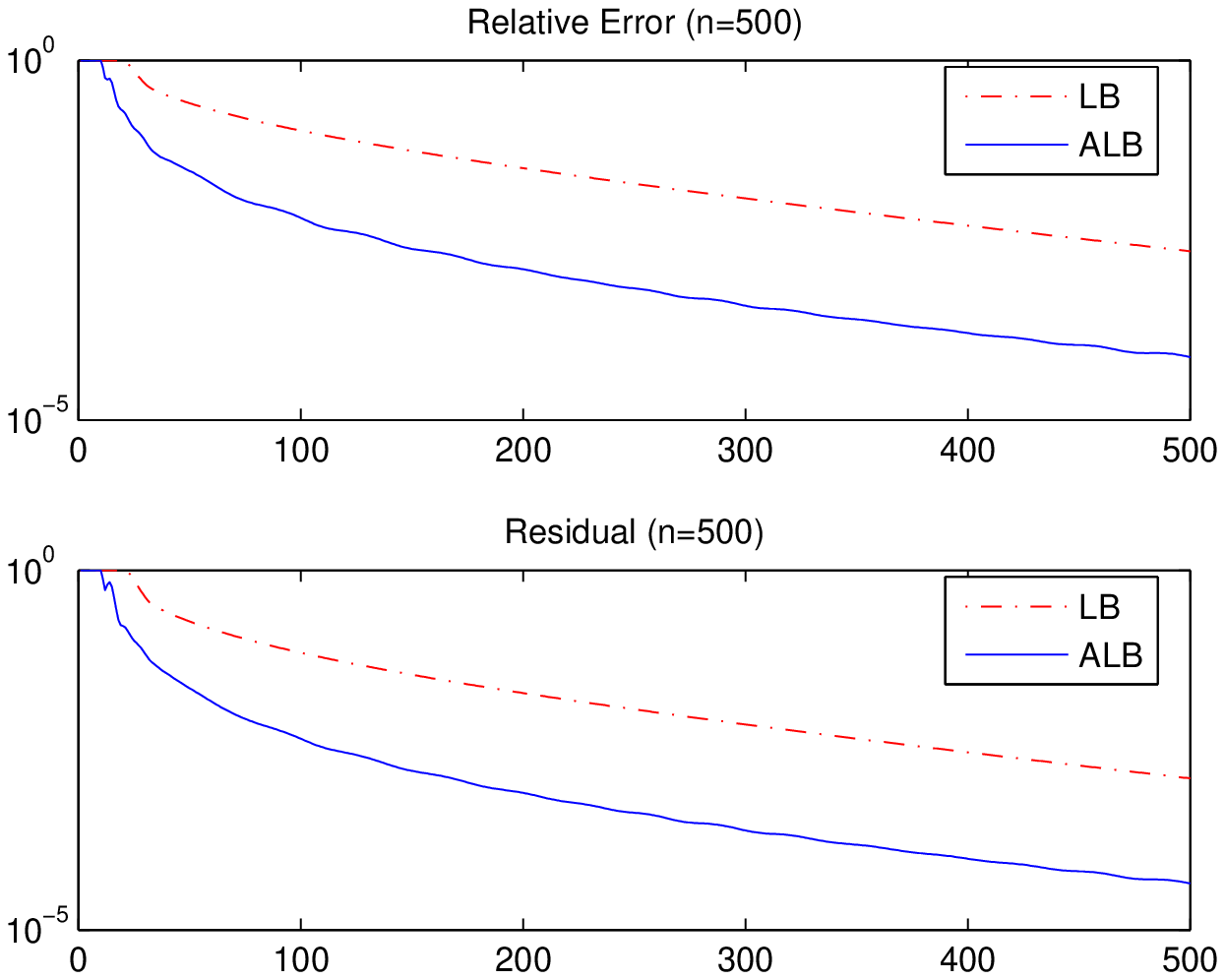}
  \caption{Comparison of LB and ALB on matrix completion problems with $\rank=10, FR=0.2$}\label{fig:MC-FR0p2}
\end{figure}
\end{center}

\begin{center}
\begin{figure}
  \hspace*{-0.1in}\includegraphics[width=0.5\textwidth]{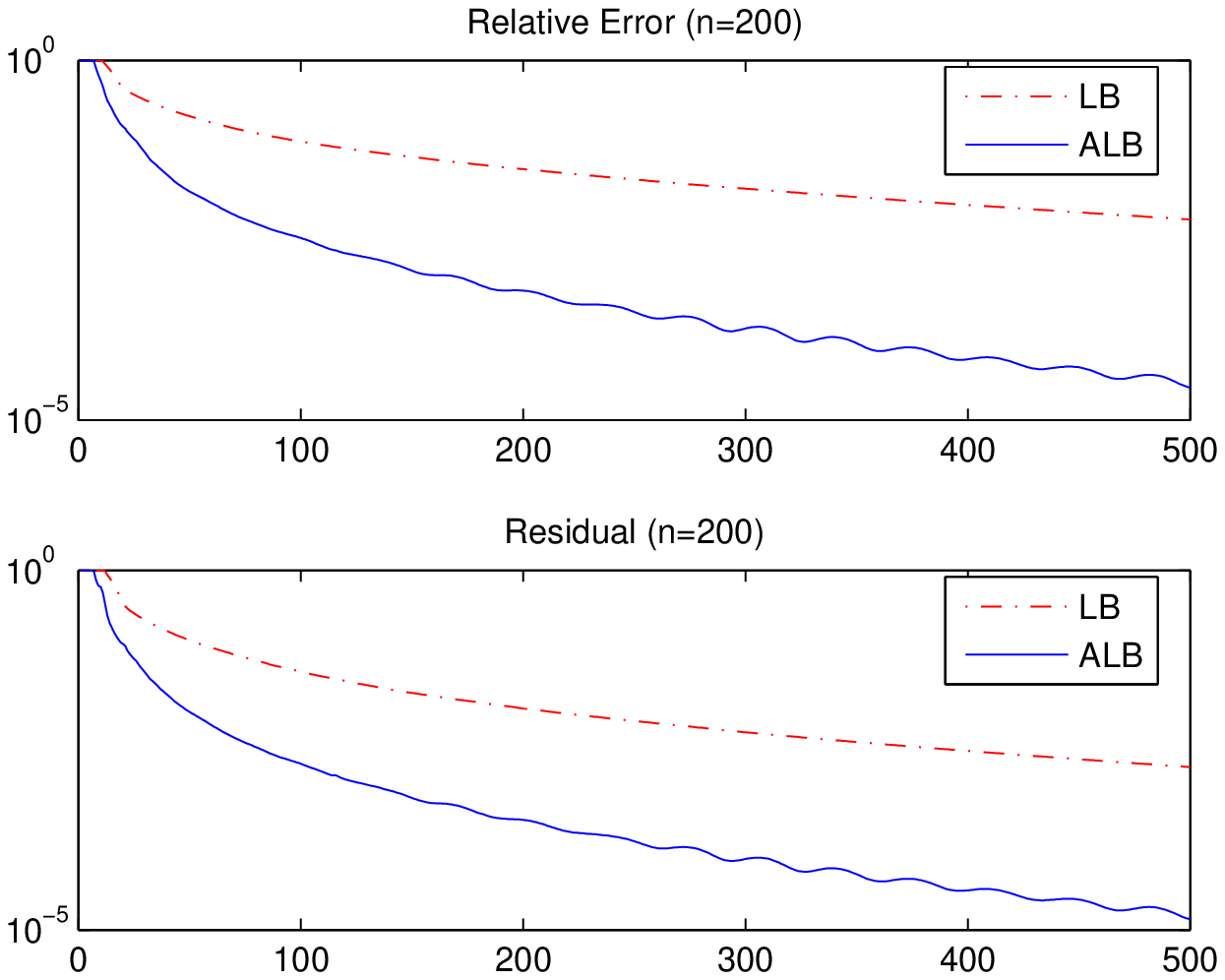}
  \hspace*{-0.1in}\includegraphics[width=0.5\textwidth]{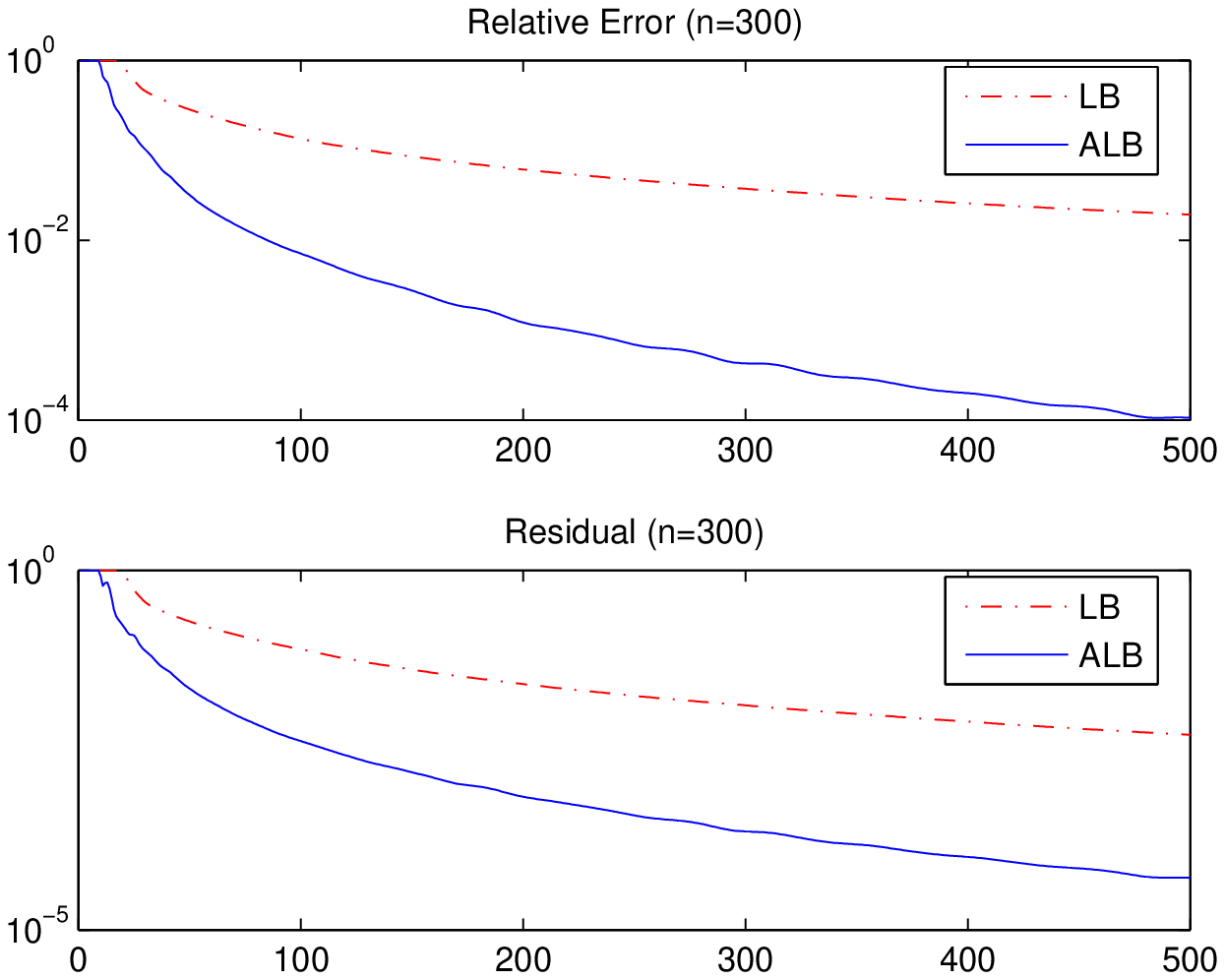}\\
  \hspace*{-0.1in}\includegraphics[width=0.5\textwidth]{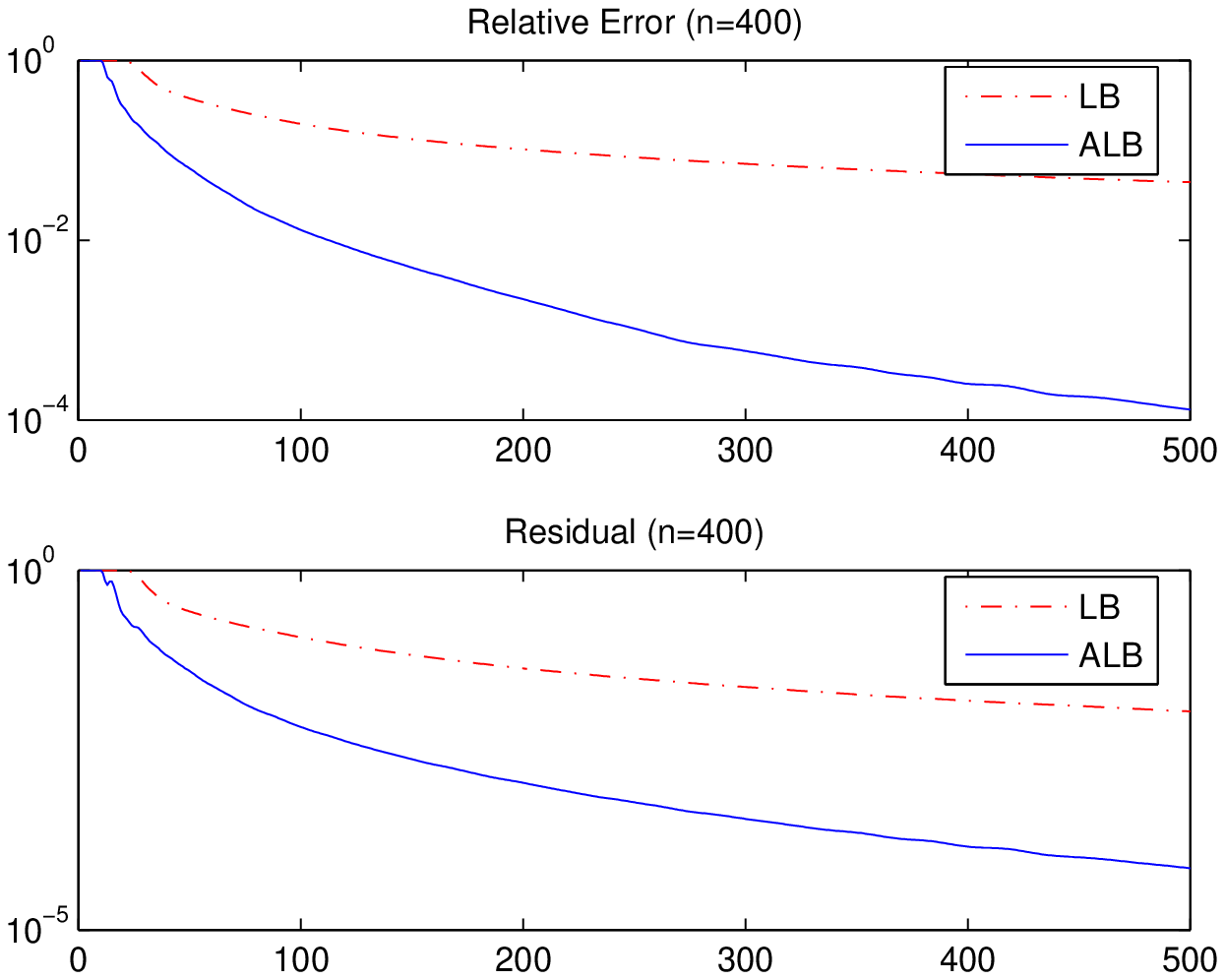}
  \hspace*{-0.1in}\includegraphics[width=0.5\textwidth]{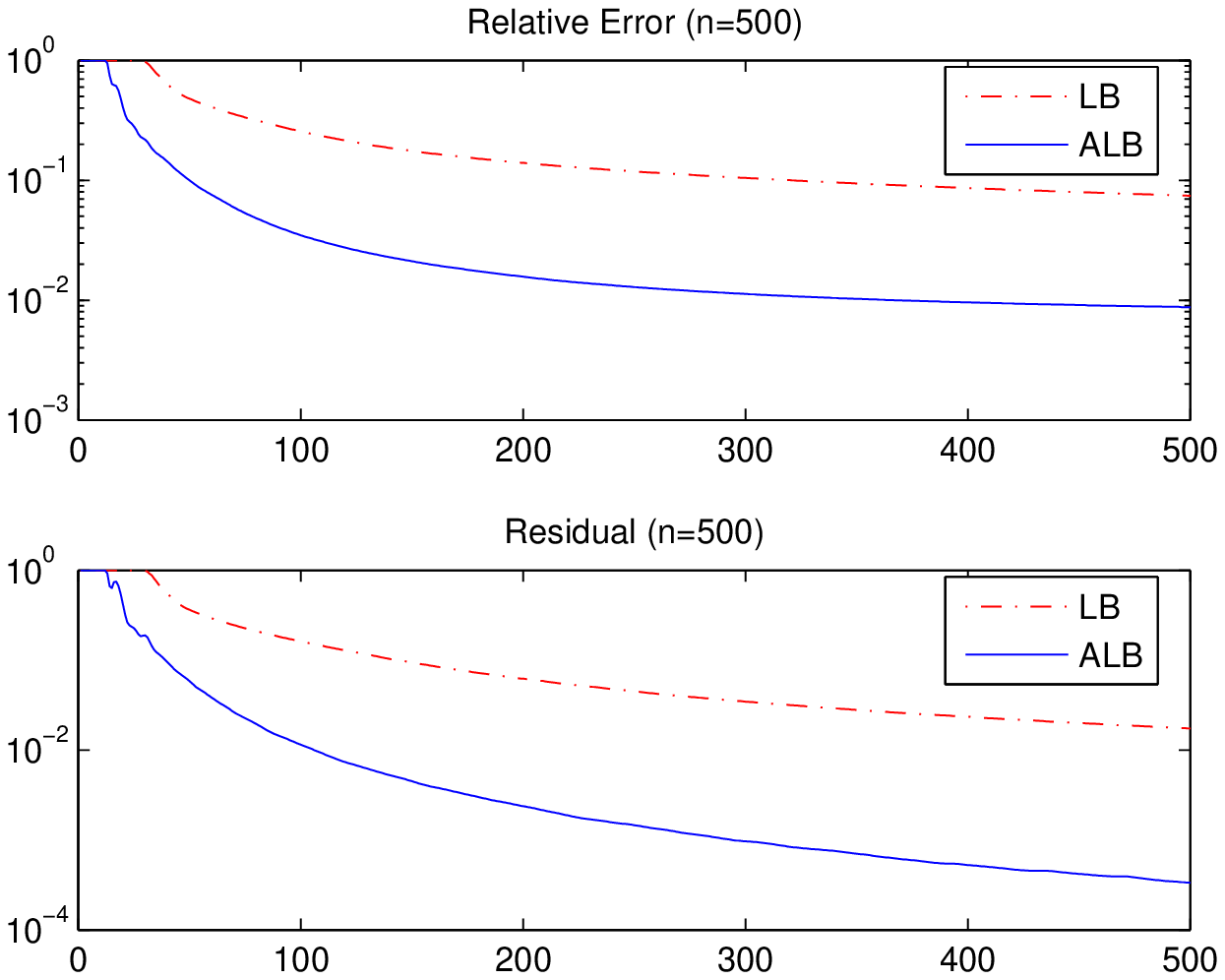}
  \caption{Comparison of LB and ALB on matrix completion problems with $\rank=10, FR=0.3$}\label{fig:MC-FR0p3}
\end{figure}
\end{center}

\section{Conclusions}\label{sec:conclusion}
In this paper, we analyzed for the first time the iteration complexity of the linearized Bregman method. Specifically, we show that for a suitably chosen step length, the method achieves a value of the Lagrangian of a quadratically regularized version of the basis pursuit problem that is within $\epsilon$ of the optimal value in $O(1/\epsilon)$ iterations. We also derive an accelerated version of the linearized Bregman method whose iteration complexity is reduced to $O(1/\sqrt{\epsilon})$ and present numerical results on basis pursuit and matrix completion problems that illustrate this speed-up.
\bibliographystyle{siam}
\bibliography{C:/Mywork/Optimization/work/reports/bibfiles/All}

\end{document}